\documentclass[11pt]{amsart}

\usepackage {amsmath, amssymb, amscd, mathrsfs, amsthm, stmaryrd,  bbm, diagbox, enumerate, slashed, graphicx, color, subfig, transparent,comment}
\usepackage[all, cmtip]{xy}
\usepackage[text={6.3in,8.8in},centering,letterpaper,dvips]{geometry}
\usepackage{url, hyperref}
\usepackage{tikz, tikz-cd}
\usetikzlibrary{decorations.markings}

\usepackage{graphics}
\usepackage{amscd}
\usepackage{mathtools} 
\usepackage{amstext} 
\usepackage{array}   
\newcolumntype{L}{>{$}l<{$}}
\newcolumntype{C}{>{$}c<{$}}

\usepackage{xspace} 
\usepackage[normalem]{ulem}

\newtheorem{thm}{Theorem}[section]
\newtheorem{theorem}[thm]{Theorem}

\newtheorem{corollary}[thm]{Corollary}

\newtheorem{lemma}[thm]{Lemma}

\newtheorem{proposition}[thm]{Proposition}

\newtheorem{conjecture}[thm]{Conjecture}

\theoremstyle{definition}

\newtheorem{definition}[thm]{Definition}

\newtheorem{question}[thm]{Question}
\theoremstyle{remark}

\newtheorem{remark}[thm]{Remark}

\newtheorem{example}[thm]{Example}

\newcommand{\td}{\widetilde}

\newcommand{\mc}{\mathcal}

\renewcommand{\theta}{\vartheta}

\newcommand{\Gabrov}{Gabrov\v{s}ek}

\newcommand{\del}{\partial}

\DeclareMathOperator{\supp}{supp}

\DeclareMathOperator{\sign}{sgn}

\newcommand{\Z}{\mathbb{Z}}
\newcommand{\Q}{\mathbb{Q}}
\newcommand{\R}{\mathbb{R}}

\newcommand{\F}{\mathbb{F}}
\newcommand{\RP}{\mathbb{RP}}

\newcommand{\tK}{\widetilde{K}}
\newcommand{\tL}{\widetilde{L}}
\newcommand{\tD}{\widetilde{D}}
\newcommand{\tW}{\widetilde{W}}
\newcommand{\tU}{\widetilde{U}}
\def\flip{\operatorname{flip}}

\newcommand{\ov}{\overline{V}}
\newcommand{\oone}{\overline{1}}
\newcommand{\ox}{\overline{X}}
\newcommand{\oa}{\overline{a}}
\newcommand{\ob}{\overline{b}}
\newcommand{\oo}{\overline{o}}

\newcommand{\gequiv}{g_{4}^{\operatorname{se}}}
\def\tSigma{\widetilde{\Sigma}}

\DeclareMathOperator{\Kh}{Kh}

\newcommand{\LH}{\mathit{LH}}
\newcommand{\s}{\mathfrak{s}}
\newcommand{\q}{\mathrm{q}} 
\newcommand{\smin}{s_{\mathrm{min}}}
\newcommand{\bo}{\bar{o}} 
\def\Kh{\mathit{Kh}}
\def\KC{\mathit{KC}}

\newcommand{\cube}{\underline{2}}
\newcommand{\Leegen}{\mathfrak{s}} 
\newcommand{\LeegenC}{\Leegen^{\mc{C}}} 
\newcommand{\LeegenCi}[1][i]{\Leegen^{\mc{C}_{#1}}}
\newcommand{\LCC}{LC^*_{\mc{C}}}
\newcommand{\LCCi}[1][i]{LC^*_{\mc{C}_{#1}}}

\newcommand{\LHCi}[1][i]{LH^*_{\mc{C}_{#1}}}

\newcommand{\unC}{\td{C}}

\newcommand{\ILtikzpic}[2][]{
	\vcenter{\hbox{\begin{tikzpicture}[#1]
				#2
	\end{tikzpicture}}}
}

\newcommand{\drawover}[2][thick]{
	\draw[line width=2mm,white] #2
	\draw[#1] #2
}

\begin{document}

\title[A Rasmussen invariant for links in $\RP^3$]{A Rasmussen invariant for links in $\RP^3$}%

\author[Ciprian Manolescu]{Ciprian Manolescu}
\address {Department of Mathematics, Stanford University\\
Stanford, CA 94305, United States of America}
\email {\href{mailto:cm5@stanford.edu}{cm5@stanford.edu}}

\author[Michael Willis]{Michael Willis}
\address{Department of Mathematics, Texas A\&M University\\
College Station, TX 77840, United States of America}
\email{\href{mailto:msw188@tamu.edu}{msw188@tamu.edu}}


\begin{abstract}
Asaeda-Przytycki-Sikora, Manturov, and Gabrov\v{s}ek extended Khovanov homology to links in $\RP^3$. We construct a Lee-type deformation of their theory, and use it to define an analogue of Rasmussen's $s$-invariant in this setting. We show that the $s$-invariant gives constraints on the genera of link cobordisms in the cylinder $I \times \RP^3$. As an application, we give examples of freely $2$-periodic knots in $S^3$ that are concordant but not standardly equivariantly concordant.
\end{abstract}

\maketitle
\section{Introduction}
In \cite{Rasmussen}, Rasmussen used Lee's deformation of Khovanov homology \cite{Kh, Lee} to define his $s$-invariant for links in $S^3$, which he employed to give a combinatorial proof of Milnor's conjecture. Rasmussen's invariant later found several other topological applications, such as Piccirillo's proof that the Conway knot is not slice \cite{Pic}. 

Khovanov homology was originally defined for links in $S^3$. Various extensions to links in other $3$-manifolds have been proposed; see \cite{Roz, Willis, APS, Gabrovsek, MWW}. Rasmussen's invariant was extended to links in $S^1 \times D^2$ in \cite{GLW}, to virtual knots in \cite{DKK, Rushworth}, and to links in connected sums of $S^1 \times S^2$ in \cite{MMSW}. In fact, Rasmussen's $s$-invariant is similar to the concordance invariant $\tau$ from knot Floer homology \cite{OStau, RasmussenThesis}, which can be defined for null-homologous knots in arbitrary $3$-manifolds. Giving a definition of the $s$-invariant in this general context remains an open problem.

In this paper we construct an $s$-invariant for links in the real projective space $\RP^3$. Our work builds on that of Asaeda-Przytycki-Sikora \cite{APS}, Manturov \cite{Manturov}, and Gabrov\v{s}ek \cite{Gabrovsek}, who defined a Khovanov-type homology for links $L \subset \RP^3$. We denote their invariant by $\Kh(L)$. (In \cite{APS} this was defined with coefficients in $\F_2$, and in \cite{Manturov, Gabrovsek} it was refined to $\Z$ coefficients.) The basic idea for defining $\Kh(L)$ is to represent $\RP^3 \setminus \text{pt}$ as an $I$-bundle over $\RP^2$, and project $L$ to $\RP^2$; the resulting diagram in $\RP^2$ is then used to build a Khovanov-type complex. Let us also note that, just as the Euler characteristic of Khovanov homology is the Jones polynomial, the Euler characteristic of the (suitably normalized) homology from \cite{APS, Manturov, Gabrovsek} is Drobotukhina's invariant from \cite{Drobo}, an analogue of the Jones polynomial for links in $\RP^3$. 

Working with coefficients in $\Q$, we will describe a Lee deformation $\LH^*$ of the Khovanov homology of links in $\RP^3$. By analogy with the calculation of Lee homology in $S^3$ in \cite{Lee}, we show the following.
\begin{theorem}
\label{thm:Lee}
Let $L \subset \RP^3$ be a link, and let $O(L)$ be the set of orientations of $L$. Then the Lee homology of $L$ is given by
$$ \LH^*(L) \cong \Q^{O(L)},$$
being generated by cycles $\s_o$, one for each $o \in O(L)$.
\end{theorem}
The cycles $\s_o$ above depend on some auxiliary data; see Section \ref{sec:Lee} and Theorem \ref{thm:Lee homology of diagram} for more precise statements.

Given an orientation $o$ on a link $L \subset \RP^3$, let $\bo$ be the opposite orientation. We define the $s$-invariant of $L$ by 
$$ s(L) = \frac{ \q([\s_o+\s_{\bar{o}}]) + \q([\s_o-\s_{\bar{o}}])}{2},$$
where $\q$ denotes the quantum filtration. We will use this to study surface cobordisms in the cylinder $I \times \RP^3$. Given such a cobordism between links $L_0$ and $L_1$, we produce a filtered chain map between the Lee complexes, of a certain filtration degree. We then obtain a constraint on the topology of the cobordism in terms of the $s$-invariant. The result is entirely analogous to the one in the usual case (in $I \times S^3$); cf.  \cite{Rasmussen}, \cite{BW}.

\begin{theorem}
\label{thm:BW}
Let $\Sigma$ be an oriented, properly embedded smooth surface in $I \times \RP^3$, with $\del \Sigma = (\{0\} \times L_0) \cup (\{1\} \times L_1)$. Suppose that every component of $\Sigma$ has a boundary component in $L_0$. Then, we have
$$ s(L_1) - s(L_0) \geq \chi(\Sigma),$$
where $\chi$ denotes the Euler characteristic.
\end{theorem}

In particular, we get a bound on the slice genus of knots in $\RP^3$. Recall that, for $K\subset S^3$, the slice genus $g_4(K)$ is the minimal genus of an oriented, connected surface in $B^4$ with boundary $K$; or, equivalently, the minimal genus of an oriented, connected cobordism in $I \times S^3$ from $K$ to the unknot. In $\RP^3$, following the terminology in \cite{MN}, we distinguish between class-0 knots and class-1 knots, according to their homology class in $H_1(\RP^3; \Z) = \Z/2.$ Observe that cobordisms $\Sigma \subset I \times \RP^3$ between two knots exist only if the knots are in the same class. For class-0 knots, we can define their slice genus in terms of cobordisms to the {\em class-0 unknot} $U_0 \subset \R^3 \subset \RP^3$ (called the affine unknot in \cite{MN}). For class-1, we will use instead cobordisms to the {\em class-1 unknot} $U_1 = \RP^1 \subset \RP^3$ (called the projective unknot in \cite{MN}).

\begin{definition}
Let $K \subset \RP^3$ be a class-$\alpha$ knot, $\alpha \in \{0,1\}$. The {\em slice genus} of $K$, denoted $g_4(K)$, is the minimal genus of a compact, oriented cobordism $\Sigma \subset I \times \RP^3$ from $K$ to the class-$\alpha$ unknot $U_{\alpha}$.
\end{definition}

The following is an immediate consequence of Theorem~\ref{thm:BW}, together with the calculation $s(U_0) = s(U_1)=0$:
\begin{corollary}
\label{cor:s}
If $K \subset \RP^3$ is a knot, then
$$ |s(K)| \leq 2g_4(K).$$ 
\end{corollary}

Let us now discuss a few calculations of the $s$-invariant. First, we consider local knots in $\RP^3$ (called affine knots in \cite{MN}), which are those contained in a copy of $\R^3$ in $\RP^3$.
\begin{theorem}
\label{thm:local}
(a) The $s$-invariant of a local knot $K \subset \R^3 \subset \RP^3$ coincides with Rasmussen's original $s$-invariant for the corresponding knot $K \subset \R^3$.

(b) If $K \subset \RP^3$ and $K' \subset S^3$ are knots, then the $s$-invariant of the connected sum $K \# K' \subset \RP^3 \# S^3 \cong \RP^3$ is given by
$$ s(K \# K') = s(K)+s(K').$$ 
\end{theorem}

In \cite{Rasmussen}, Rasmussen computed his $s$-invariant for all positive knots in $S^3$. One can also define positive knots in $\RP^3$, to be those that admit a diagram in $\RP^2$ with only positive crossings. Here is the analogue of Rasmussen's calculation.
\begin{theorem}
\label{thm:positive}
Let $K \subset \RP^3$ be a positive knot. Let $n$ be the number of crossings in a positive diagram of $K$, and let $r$ be the number of circles in the oriented resolution of that diagram. Then:
$$ s(K) = n-r+1.$$
\end{theorem}

The $s$-invariant and the slice genus of some $K \subset \RP^3$ can be compared with those of its lift to the $2$-fold cover $S^3 \to \RP^3$. Let us focus on the case of class-$1$ knots $K$, for which this lift (denoted $\tK)$ is a knot rather than a link. Such knots $\tK \subset S^3$ are called {\em freely $2$-periodic}, and were studied for example in \cite{BM}, \cite{MN}. The quantity $2g_4(K)$ is easily seen to be the same as the {\em standardly equivariant slice genus} $\gequiv(\tK)$, the minimal genus of an oriented, properly and smoothly embedded  surface $\tSigma \subset B^4$ with boundary $\tK$ such that $\tSigma$ is invariant under the involution $x \mapsto -x$ on $B^4$. (This is different from the {\em equivariant slice genus} of $\tK$, for which more general involutions are allowed; compare \cite{BM}.) We have
$$ g_4(\tK) \leq 2g_4(K) =\gequiv(\tK).$$ 

Let us compare the lower bound on $\gequiv(\tK)$ from Corollary~\ref{cor:s},
\begin{equation}
\label{eq:bound1}
|s(K)| \leq \gequiv(\tK),
\end{equation}
 with Rasmussen's bound 
 \begin{equation}
\label{eq:bound2}
 |s(\tK)/2| \leq g_4(\tK) \leq \gequiv(\tK).
 \end{equation}
  In the case of positive knots, by applying the results in \cite{Rasmussen} and Theorem~\ref{thm:positive} we get that the two bounds are the same; in fact, we have
\begin{equation}
s(\tK)/2 = s(K) = g_4(\tK) = \gequiv(\tK).
\end{equation}

Theorem~\ref{thm:local} implies that we also have $s(\tK) = 2s(K)$ for class-1 knots of the form $U_1 \# K$, where $K \subset S^3$.

On the other hand, computer calculations show that $s(\tK) \neq 2s(K)$ in the following two examples.

\begin{example}
\label{ex:K1}
Consider the knot $K_1 \subset \RP^3$ shown in Figure~\ref{fig:K1}. Its lift $\tilde K_1 \subset S^3$ is the freely $2$-periodic knot $12n403$. We have
$$ s(K_1)=0, \ \ s(\tK_1) = 2,$$
so in this case our genus bound \eqref{eq:bound1} is weaker than the bound \eqref{eq:bound2} from the original $s$-invariant. It turns out that $\gequiv(\tK_1) = g_4(\tK_1) =2$.
\end{example}

\begin{figure}
{
   \def\svgwidth{1.8in}
   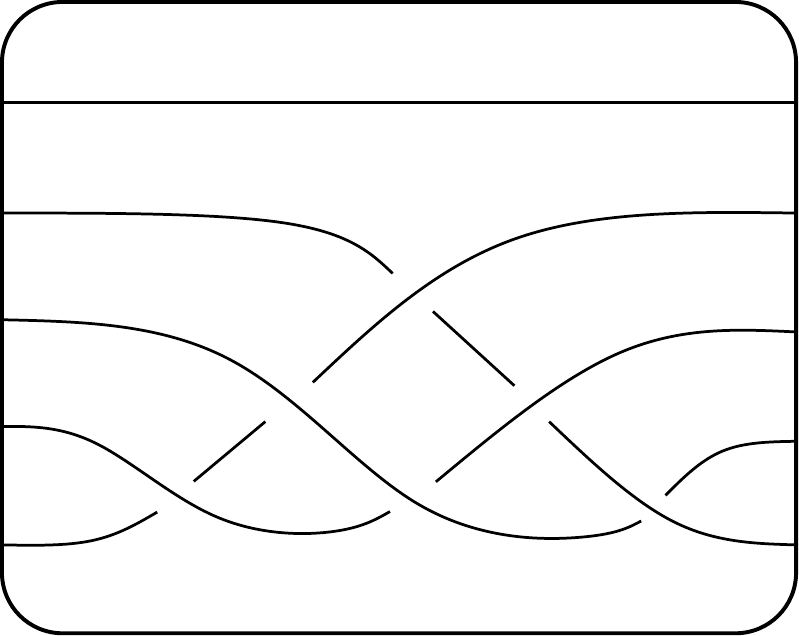
}
\caption{A knot $K_1$ in $\RP^3$, presented through its projection to $\RP^2$. The antipodal points on the boundary of the disk are identified to form $\RP^2$.}
\label{fig:K1}
\end{figure}

\begin{example}
\label{ex:K2}
Let  $K_2 \subset \RP^3$ be the knot from Figure~\ref{fig:K2}. Its lift $\tilde K_2 \subset S^3$ is the freely $2$-periodic knot $14n14256$. We have
$$ s(K_2) = 2, \ \ \ s(\tK_2) = 2.$$
Thus, our new bound \eqref{eq:bound1} reads $\gequiv(\tK_2) \geq 2$, which looks stronger than the bound $\gequiv(\tK_2) \geq s(\tK_2)/2=1$, albeit it is not really stronger because we know that $\gequiv(\tK_2)=2g_4(K_2)$ is even. Nevertheless, this is an instructive example because it turns out that the slice genus and the standardly equivariant slice genus are different:
 $$g_4(\tK_2) =1 < 2=\gequiv(\tK_2).$$
\end{example}

\begin{figure}
{
   \def\svgwidth{1.8in}
   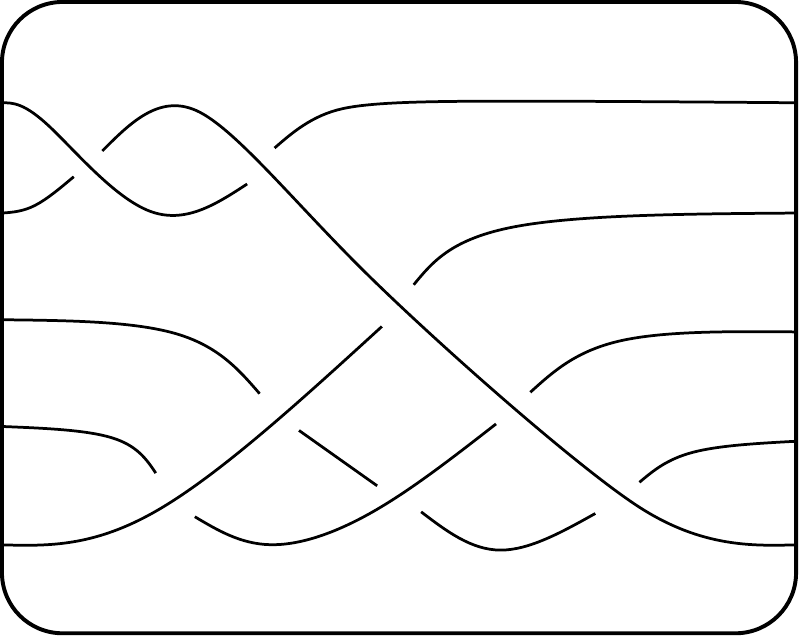
}
\caption{Another knot, $K_2$, in $\RP^3$.}
\label{fig:K2}
\end{figure}

We say that two freely periodic 2-knots are {\em standardly equivariantly concordant} if there is a smooth annular cobordism between them that is equivariant under the involution on $I \times S^3$ given by the identity times the antipodal map; this is equivalent to asking for their quotients in $\RP^3$ to be concordant in $I \times \RP^3$. By using the properties of our $s$-invariant, we obtain the following result.

\begin{theorem}
\label{thm:eqconc}
Let $K \subset \RP^3$ be either the knot $K_1$ from Example~\ref{ex:K1} or the knot $K_2$ from Example~\ref{ex:K2}. Let $\tK$ be the lift of $K$ to $S^3$, and view $\tK' = (-\tK) \# \tK \# \tK$ as the lift of $K' = (-K) \# \tK \subset \RP^3 \# S^3 \cong \RP^3.$ Then, the freely $2$-periodic knots $\tK$ and $\tK' = (-\tK) \# \tK \# \tK$ are concordant but not standardly equivariantly concordant.
\end{theorem}

\medskip

\textbf{Organization of the paper.} In Section~\ref{sec:Defining deformed cx and Lee cx} we define the Lee deformation of the Khovanov complex in $\RP^3$. In Section~\ref{sec:Lee} we compute the Lee homology, proving Theorem~\ref{thm:Lee}.  In Section~\ref{sec:s-invt defined} we define our $s$-invariant for links in $\RP^3$ and study some of its basic properties, proving Theorem~\ref{thm:local}. In Section~\ref{sec:cobordisms} we study how the Lee generators behave under cobordisms, and prove Theorem~\ref{thm:BW}, Corollary~\ref{cor:s}, and Theorem~\ref{thm:local}. In Section~\ref{sec:app} we give some properties and applications, including Theorems~\ref{thm:positive} and \ref{thm:eqconc}. At the end we also discuss some open problems.

\medskip

\textbf{Acknowledgements.} We would like to thank Keegan Boyle, Daren Chen, Irving Dai, Anthony Licata, and Maggie Miller for helpful conversations. We are also grateful to an anonymous referee for many comments that improved the paper, and especially for providing us with the proof of Theorem~\ref{thm:no extra s-invariants}.

Part of this work was supported by NSF Grant No.1440140, while the authors were in residence at the Simons Laufer Mathematical Sciences Institute in Berkeley, California, during Fall 2022. The work was also partially supported by the first author's NSF grant DMS-2003488 and Simons Investigator award.

\section{The deformed complex and the Lee complex}\label{sec:Defining deformed cx and Lee cx}
In \cite{Gabrovsek}, \Gabrov~ constructs a Khovanov chain complex for framed, oriented links in $\RP^3$.  After adjusting grading conventions to remove the dependency on the framing, we will deform the differentials in his construction to arrive at a deformed Khovanov complex for oriented links in $\RP^3$, in a manner analogous to the construction of the Lee complex for annular links defined in \cite{GLW}.  Before we begin we set some terminology for the links under consideration.

Let us identify $\RP^3 \setminus \text{pt}$ with the non-trivial $I$-bundle $E$ over $\RP^2$, given by
$$E=(D^2 \times I) / (x, t) \sim (-x, 1-t), \forall x \in D^2, \forall t \in [0,1].$$ 

\begin{definition}
A {\em diagram} $D$ for a link $L \subset \RP^3$ is a tangle diagram in the oriented disk $D^2$, with the ends of the tangle consisting of pairs of antipodal points. (See Figures~\ref{fig:K1} and \ref{fig:K2} for examples.) We require that when we lift the diagram $D$ to $E \subset \RP^3$ so that the overstrands go on top, we obtain a link isotopic to $L$.
\end{definition}

Note that after we identify all the antipodal points on the boundary of $D^2$, we can view $D$ as a diagram in $\RP^2$. However, defining it as a subset of the oriented $D^2$ will allow us to talk about understrands and overstrands.

\begin{definition}\label{def:class-0 class-1 local}
A link $L\subset \RP^3$ is \emph{class-0} (respectively \emph{class-1}) if $[L]=0$ (respectively $[L]=1$) as an element of $H_1(\RP^3;\Z)\cong\Z/2\Z$.  The link $L$ is called \emph{local} if it is contained in a ball $B^3\subset\RP^3$ (and thus must also be class-0).  A \emph{local} diagram $D\subset\RP^2$ for a local link is a link diagram contained within a 2-disc $B^2\subset\RP^2$.
\end{definition}
See \cite{MN}, where local knots are referred to as \emph{affine knots}.

Now let $D\subset \RP^2$ be a link diagram with $n$ crossings for some link $L$ in $\RP^3$ (we are ignoring any orientation for $L$ or $D$ at present).  After choosing an ordering of the crossings in $D$, we arrive at (unoriented) resolutions $D_v\subset \RP^2$ for each vertex $v$ in the cube $\cube^n:=(0\rightarrow 1)^n$ by following the usual conventions for 0-resolutions and 1-resolutions of crossings as indicated in Figure \ref{fig:resolutions}.  Each resolution $D_v$ consists of some number of homologically trivial circles, together with at most one homologically essential circle as determined by the following lemma.

\begin{lemma}\label{lem:ess circ in Dv iff L is class-1}
A homologically essential circle exists in a resolution $D_v$ if and only if $D$ is the diagram for a class-1 link $L$.
\end{lemma}
\begin{proof}
An inductive argument on crossings shows that every resolution $D_v\subset\RP^2$ of $D$ must represent the same class in $H_1(\RP^2;\Z)\cong \mathbb{Z}/2\mathbb{Z}$ as $D$, which itself must represent the same class as $L$ in $H_1(\RP^3;\Z)\cong \mathbb{Z}/2\mathbb{Z}$.
\end{proof}

Now in order to build the chain groups at each vertex, and the differentials between them, we follow \Gabrov~ \cite{Gabrovsek} and introduce some auxiliary choices.

\begin{figure}
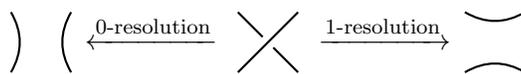

\[
\,\ILtikzpic[scale=.8]{
\draw[thick] (0,1) to[out=-60,in=60] (0,0);
\draw[thick] (1,1) to[out=-120,in=120] (1,0);
}\,
\xleftarrow{\text{0-resolution}}
\,\ILtikzpic[scale=.8]{
\draw[thick] (0,1) -- (1,0);
\drawover{ (0,0) -- (1,1); }
}\,
\xrightarrow{\text{1-resolution}}
\,\ILtikzpic[scale=.8]{
\draw[thick] (0,1) to[out=-30,in=-150] (1,1);
\draw[thick] (0,0) to[out=30,in=150] (1,0);
}\,
\]
\caption{Each crossing in a link diagram $D\subset\RP^2$ can be resolved in one of two ways as indicated.  Thus a vertex $v\in\cube^n$ determines a resolution $D_v$ consisting of circles in $\RP^2$.}
\label{fig:resolutions}
\end{figure}

\begin{definition}\label{def:auxiliary diagram choices}
Given a fixed (unoriented) link diagram $D\subset\RP^2$, a set of \emph{auxiliary diagram choices} for $D$ consists of the following data.
\begin{enumerate}
\item For each crossing $c$ in $D$, we choose a \emph{local orientation} for the two strands at $c$, meaning an orientation for each strand that need only be locally consistent.  See Figure \ref{fig:local orientations} for an example.
\item For each $v\in\cube^n$, we choose an ordering of the circles in $D_v$.
\item For each $v\in\cube^n$, we choose an orientation for $D_v$ (that is to say, an orientation for each circle in $D_v$).
\end{enumerate}
\end{definition}

\begin{figure}
	\[\begin{tikzpicture}
	\draw[thick,->] (1,0) -- (0,1);
	\draw[thick,->] (0,3) -- (1,2);
	\drawover[thick,<-]{(0,0) -- (1,1);}
	\drawover[thick,->]{(1,3) -- (0,2);}
	\draw[thick] (0,1) to[out=135,in=-135] (0,2);
	\draw[thick] (1,1) to[out=45,in=-45] (1,2);
	\draw[thick] (0,0) to[out=-135,in=-90] (-1,1.5) to[out=90,in=135] (0,3);
	\draw[thick] (1,0) to[out=-45,in=-90] (2,1.5) to[out=90,in=45] (1,3);
	\end{tikzpicture}\]
	\caption{An \emph{unoriented} diagram for the Hopf link, with each crossing given a \emph{local} orientation.  Note that these local orientations have no global consistency requirements.}
	\label{fig:local orientations}
\end{figure}
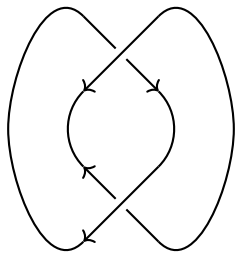

\begin{remark}\label{rmk:new choices vs Gabs choices}
Note that an orientation on the link diagram $D\subset\RP^2$ induces a choice of local orientations at each crossing $c$ that is globally consistent rather than simply locally consistent.  In \cite{Gabrovsek}, \Gabrov~ uses an orientation of $D$ in place of a choice of local orientations; however there is a small gap in his proof that the resulting differential squares to zero.  The added flexibility of allowing orientations at crossings that are only locally consistent will allow us to fill this gap, and will also simplify the presentation of the skein exact sequence for a crossing in Proposition \ref{prop:skein exact seq}.

In addition, \Gabrov~ in \cite{Gabrovsek} demands that any homologically essential circle in each $D_v$ is ordered last, but this is for notational convenience only.
\end{remark}

Let $R$ be a commutative ring. Define two free rank two $R$-modules
$$V=\langle 1,X\rangle, \ \ \ \ov=\langle \oone,\ox\rangle.$$
To each oriented, ordered resolution $D_v$, we assign a tri-graded (unnormalized) chain group $\unC(D_v)$ which is a tensor product of factors of $V$ for each trivial circle, and $\ov$ for the essential circle (if present).  The ordering of the circles determines the ordering of this tensor product.  The tri-grading, denoted $(i,j,k)$, is defined on a generator in $\unC(D_v)$ as follows:
\begin{align*}
i&=|v|,\\
j &=|v| + \#(\text{circles labeled with $1$ or $\oone$}) - \# \text{(circles labeled with $X$ or $\ox$)},\\
k &= \#\text{(circles labeled with  $\oone$)} - \# \text{(circles labeled with  $\ox$),}
\end{align*}
where $|v|$ denotes the norm of $v\in\cube^n$, i.e. the sum of the entries of $v$.  

We may think of $i$ as a homological grading, and $j$ as an internal quantum grading (both to be globally renormalized at a later point in the construction).  The grading $k$ is similar to the annular grading in annular Khovanov and Lee homology \cite{GLW}, and tracks the generator on an essential circle (if present).  Since there can be at most one such circle, we have $k\in\{-1,0,1\}$.

Now pick elements $s, t \in R$. To each edge $e:u\rightarrow v$ of the cube $\cube^n$, we assign a differential
\begin{equation}
\label{eq:dele}
\del^e = \del^e_0 +  {s \del^e_-} +  {st \Phi^e_0} +  {t \Phi^e_+},\end{equation}
where $\del^e_0$ will be a suitable generalization of the Khovanov differential assigned by \Gabrov~ in \cite{Gabrovsek} (see Remark \ref{rmk:new choices vs Gabs choices}).  The values of these various components of the differential depend on the type of edge $e$ that is under consideration. 

\begin{definition}
We say that the edge $e: u \to v$ corresponds to:
\begin{itemize}
\item a {\em 1-2 bifurcation} if $v$ is obtained from $u$ by splitting one circle into two;
\item a {\em 2-1 bifurcation} if $v$ is obtained from $u$ by combining two circles into one;
\item a {\em 1-1 bifurcation} if $v$ is obtained from $u$ by turning a circle into another circle, as in Figure \ref{fig:1-1 bifurcation}.
\end{itemize}
\end{definition}

\begin{figure}
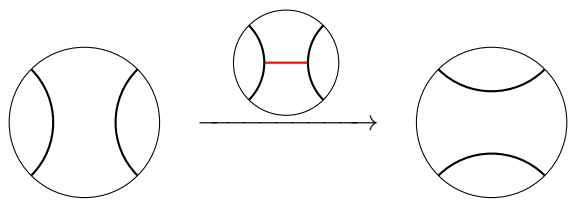

\[\ILtikzpic{
\draw[thick] (45:1) to[out=-135,in=135] (-45:1);
\draw[thick] (135:1) to[out=-45,in=45] (-135:1);
\draw (0,0) circle (1);
	}
\quad\xrightarrow{\quad\ILtikzpic[scale=.7]{
\draw[thick] (45:1) to[out=-135,in=135] (-45:1);
\draw[thick] (135:1) to[out=-45,in=45] (-135:1);
\draw (0,0) circle (1);
\draw[thick,red] (-.4,0)--(.4,0);
	}\quad}\quad
\ILtikzpic{
\draw[thick] (45:1) to[out=-135,in=-45] (135:1);
\draw[thick] (-45:1) to[out=135,in=45] (-135:1);
\draw (0,0) circle (1);
}
\]
\caption{A 1-1 bifurcation, corresponding to an unorientable saddle in $\RP^2\times I$. As in Figure~\ref{fig:K1}, the boundary of each circle is quotiented by the antipodal map.}
\label{fig:1-1 bifurcation}
\end{figure}

\begin{remark}\label{rmk:No 1-1 bifurcations for class-1}
Observe that 1-1 bifurcations do not appear for the usual planar diagrams of links in $S^3$. Moreover, in $\RP^3$, they can only appear for diagrams of non-local class-0 links. 
\end{remark}

Going back to \eqref{eq:dele}, if the edge $e$ corresponds to a 1-1 bifurcation, we set each of the components of $\del^e$ to be identically zero.  Otherwise, $e$ corresponds to either a 2-1 bifurcation  (multiplication $m$) or a 1-2 bifurcation (comultiplication $\Delta$), and then the values of these components, ignoring signs, are described by the following table.

\begin{equation}\label{eq:differential table}
\setlength{\extrarowheight}{2pt}{
\begin{tabular}{||CCCCCC||CCCCCCL||}
\hline
m: & V & \otimes & V & \to & V & \Delta : & V & \to & V & \otimes & V & \\
\hline
& 1 & \otimes & 1 &\mapsto & 1 & & 1 & \mapsto & 1 & \otimes & X &+ X \otimes 1\\
& 1 & \otimes & X &\mapsto & X & & X & \mapsto & X & \otimes & X &+  {st(1 \otimes 1)} \\
&X &\otimes &1  &\mapsto & X & & & & & & &\\
&X & \otimes & X  &\mapsto &  {st1} & & & & & & &\\
\hline
\hline 
m: & \ov & \otimes & V & \to & \ov & \Delta: & \ov& \to&\ov & \otimes & V &\\
\hline
& \oone & \otimes & 1 &\mapsto & \oone & & \oone & \mapsto &\oone & \otimes & X &+  {s( \ox \otimes 1)}  \\
& \oone & \otimes & X &\mapsto &  {s\ox} & & \ox &\mapsto & \ox& \otimes & X &+  {t (\oone \otimes 1)}\\
& \ox &\otimes &1  &\mapsto & \ox & & & & & & &\\
& \ox & \otimes & X  &\mapsto &  {t\oone} & & & & & & &\\
\hline
\hline 
m: & V & \otimes & \ov & \to & \ov & \Delta: & \ov& \to & V& \otimes & \ov& \\
\hline
& 1 & \otimes & \oone &\mapsto & \oone & & \oone & \mapsto & X& \otimes& \oone &+  {  s(1 \otimes \ox) }\\
& X & \otimes & \oone &\mapsto &  {s\ox} & & \ox & \mapsto & X& \otimes &\ox &+  { t(1 \otimes \oone)}  \\
& 1 &\otimes &\ox  &\mapsto & \ox & & & & & & &\\
&X & \otimes & \ox  &\mapsto &  {t\oone} & & & & & & &\\
\hline

\end{tabular}
}
\end{equation}

The reader can verify that, just as for the annular Khovanov-Lee complex of \cite{GLW}, the components of a non-zero differential change the $(i, j, k)$ degrees by
\begin{itemize}
\item $\deg(\del^e_0) = (1, 0, 0), $
\item ${ \deg(\del^e_-) = (1, 0, -2)}, $
\item ${ \deg(\Phi^e_0) = (1, 4, 0)}, $ 
\item ${ \deg(\Phi^e_+) = (1, 4, 2)}.$ 
\end{itemize}
Note that the local orientations of crossings in $D$ and the orientations of the various $D_v$ have not yet been used.

Finally, to decide the sign of any component of the differential $\del^e:\unC(D_u)\rightarrow \unC(D_v)$, we implement the following three rules (these are not applied to 1-1 bifurcations, since those give the zero map).
Let $c$ denote the crossing in $D$ corresponding to the edge $e:u\rightarrow v$, and assign cardinal directions to the neighborhood of $c$ so that the two strands are locally oriented northwest and northeast as in Figure \ref{fig:cardinal directions}.

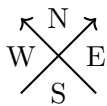
\begin{figure}
	\[\begin{tikzpicture}
	\draw[thick,->] (0,0) -- (1,1);
	\draw[thick,->] (1,0) -- (0,1);
	\node (N) at (.5,1) {N};
	\node (E) at (1,.5) {E};
	\node (S) at (.5,0) {S};
	\node (W) at (0,.5) {W};
	\end{tikzpicture}\]
	\caption{A local orientation at a crossing (which strand is above the other is not relevant) determines local cardinal directions such that the strands are oriented northwest and northeast.}
	\label{fig:cardinal directions}
\end{figure}

\begin{enumerate}
\item[(P)] Permutation Rule (uses the local orientation of $c$ and the ordering of $D_u,D_v$):  For $D_u$ (respectively $D_v$), consider the permutation $\sigma_u$ (respectively $\sigma_v$) of the circles which forces the circle in $D_u$ (resp. $D_v)$ to the west/north of $c$ to be first, forces the other circle near $c$ (if it is a separate circle) to be second, and keeps the other circles (the ones not near $c$) to be in the same order as in the given auxiliary diagram choice. Then all components of $\del^e$ are multiplied by $(\sign(\sigma_u)\cdot\sign(\sigma_v^{-1}))$. 

\item[(O)] Far Orientations Rule (uses orientations of $D_u,D_v$ far from $c$): For any generator $g\in \unC(D_u)$ (labelling of circles in $D_u$ by 1's or $X$'s), let $X^{far}_{e}(g)$ denote the number of circles which are:
  \begin{itemize}
	\item disjoint from the neighborhood of $c$ (and thus equivalent in $D_u$ and $D_v$);
	\item labeled $X$ (or $\ox$) according to the generator $g$; and
	\item oriented differently in $D_v$ compared to $D_u$.
  \end{itemize}
Then the component $\del^e(g)$ is multiplied by $(-1)^{X^{far}_e(g)}$.

\item[(C)] Nearby Consistency Rule (uses orientations of $D,D_u,D_v$ near $c$): For each circle $C$ in $D_u$ or $D_v$ passing through the neighborhood of $c$, the orientation of $C$ can be compared to the local orientation of $c$.  If these orientations match on the northeast corner, we call $C$ \emph{consistent} (else inconsistent).  If they match on the southwest corner, we call $C$ \emph{inconsistent} (else consistent). See Figure \ref{fig:circles consistent}.  Note that, since we assumed we are not at a 1-1 bifurcation, if $C$ contains both the northeast and the southwest corners, then being consistent at one corner is the same as being consistent at the other.

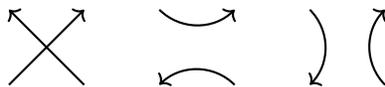
\begin{figure}
	\[\begin{tikzpicture}
	\draw[thick,->] (0,0)--(1,1);
	\draw[thick,->] (1,0)--(0,1);
	\draw[thick,->] (2,1) to[out=-45,in=-135] (3,1);
	\draw[thick,->] (3,0) to[out=135,in=45] (2,0);
	\draw[thick,->] (4,1) to[out=-45,in=45] (4,0);
	\draw[thick,->] (5,0) to[out=135,in=-135] (5,1);
	\end{tikzpicture}\]
	\caption{A local orientation for a crossing determines which orientations for nearby circles are considered locally consistent as shown.  The pre-chosen orientations for the circles in each resolution are compared with these to determine signs according to rule (C).}
	\label{fig:circles consistent}
\end{figure}

Now for any generator $g\in \unC(D_u)$, let $h\in \unC(D_v)$ be a generator which appears as a summand in $\del^e(g)$ (if $e$ corresponds to multiplication $m$, there is one such $h$, while if $e$ corresponds to comultiplication $\Delta$, there are two such $h$), and let $X^{near}_e(g,h)$ denote the number of circles near $c$ which are:
  \begin{itemize}
	\item labeled $X$ (or $\ox$) in either $D_u$ by $g$ or $D_v$ by $h$; and
	\item are inconsistent.
  \end{itemize}
Then the generator $h$ in the sum $\del^e(g)$ is multiplied by $(-1)^{X^{near}_e(g,h)}$.
\end{enumerate}

Altogether then, signs are decided by considering permutations of orderings (P), together with checking orientations of circles labeled by $X$ (or $\ox$) in both $D_u$ and $D_v$.  Such circles which are far from the crossing $c$ are counted if they change orientation (O), while such circles nearby are counted if they compare `unfavorably' with the local orientation of $c$ (C).  These rules may be applied to each component of $\del^e$ separately, or to all of them at once.

\begin{definition}\label{def:unnorm chain complex}
Given a link diagram $D\subset\RP^2$ (oriented or not) with $n$ crossings and a set of auxiliary diagram choices as in Definition \ref{def:auxiliary diagram choices}, we define the following notations:
\[ \unC(D) := \bigoplus_{v\in\cube^n} \unC(D_v), \quad \del := \sum_{\text{edges $e$ in $\cube^n$}} \del^e.\]
\end{definition}

We will refer to $\del$ as a differential in the following lemma, although we do not yet know that $\del^2=0$.  Compare the following to \cite[Lemma 6.2]{Gabrovsek}, which as stated concerns changes in global link diagram orientation, but whose proof hinges on changing local orientation at a single crossing at a time.

\begin{lemma}\label{lem:change local ori gives chain iso}
Fix a link diagram $D\subset\RP^2$, and let $(\unC(D),\del)$ and $(\unC(D)',\del')$ denote two $R$-modules with differentials built as in Definition \ref{def:unnorm chain complex}, which differ only in the auxiliary choice of local orientations at crossings.  Then there exists a natural\footnote{In this paper, whenever we use the word {\em natural}, we mean independent of any other choices.} isomorphism $\unC(D)\xrightarrow{\phi}\unC(D)'$ such that $\phi\circ\del = \del'\circ\phi$.  Moreover, if $\unC(D)''$ is the complex using a third choice of local orientations at crossings, the following diagram of natural isomorphisms commutes.
\begin{equation}\label{eq:naturality of local ori change}
	\begin{tikzcd}
		\unC(D) \ar[r,"\phi"] \ar[rr,bend left,"\phi"] &  \unC(D)' \ar[r,"\phi"] &  \unC(D)''
	\end{tikzcd}
\end{equation}
\end{lemma}
\begin{proof}
We begin by considering the effect of changing the local orientation at one fixed crossing $c$.  Doing so affects the signs in the component of the differential \begin{equation}\label{eq:differential for crossing}
	\del^c:=\sum_{e\sim c} \del^e
\end{equation}
where $e\sim c$ indicates that we are summing over every edge which corresponds to a saddle at $c$.  We claim that, for \emph{any} change of local orientation at $c$, this full sum $\del^c$ changes at most by a single sign.  Specifically, $\del^c$ is negated if and only if the understrand of $c$ is changed (otherwise $\del^c$ is not affected at all).

Up to rotation by $\pi$ there are 2 possible local orientations for $c$ before the change, corresponding to $c$ starting out locally positive or locally negative.  Once this initial choice is made however, this symmetry is broken and there are three possible choices for how to change the local orientation of $c$ (we can change the overstrand, the understrand, or both).  Thus there are six total cases to check.  We analyze one of them below.

Suppose $c$ is initially locally positive, and we change the local orientation of the understrand.  The situation is illustrated in Figure \ref{fig:local orientation change}.  Changing the understrand affects the signs of every component $\del^e$ in Equation \eqref{eq:differential for crossing}.  Some of these edges will correspond to multiplication $m$, and others to comultiplication $\Delta$.

\begin{figure}
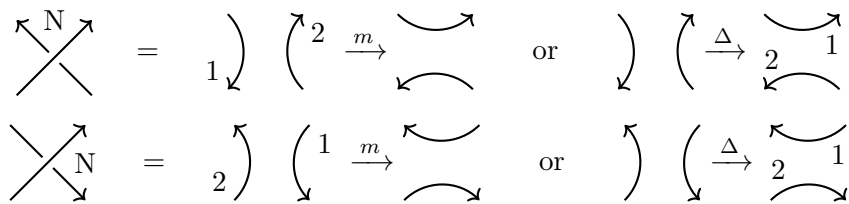

\[
\ILtikzpic{
	\draw[thick,->] (1,0)--(0,1);
	\drawover[thick,->] {(0,0)--(1,1);}
	\node at (.5,1){N};
	}
\quad = \quad
\ILtikzpic{
	\draw[thick,->] (0,1) to[out=-45,in=45] (0,0);
	\draw[thick,->] (1,0) to[out=135,in=-135] (1,1);
	\node at (-.2,.25) {1};
	\node at (1.2,.75) {2};
	}
\xrightarrow{\, m \,}
\ILtikzpic{
	\draw[thick,->] (0,1) to[out=-45,in=-135] (1,1);
	\draw[thick,->] (1,0) to[out=135,in=45] (0,0);
	}
\qquad \text{or} \qquad
\ILtikzpic{
	\draw[thick,->] (0,1) to[out=-45,in=45] (0,0);
	\draw[thick,->] (1,0) to[out=135,in=-135] (1,1);
}
\xrightarrow{\, \Delta \,}
\ILtikzpic{
	\draw[thick,->] (0,1) to[out=-45,in=-135] (1,1);
	\draw[thick,->] (1,0) to[out=135,in=45] (0,0);
	\node at (.1,.4) {2};
	\node at (.9,.6) {1};
}
\]
\[
\ILtikzpic{
	\draw[thick,<-] (1,0)--(0,1);
	\drawover[thick,->] {(0,0)--(1,1);}
	\node at (1,.5){N};
}
\quad = \quad
\ILtikzpic{
	\draw[thick,<-] (0,1) to[out=-45,in=45] (0,0);
	\draw[thick,<-] (1,0) to[out=135,in=-135] (1,1);
	\node at (-.2,.25) {2};
	\node at (1.2,.75) {1};
}
\xrightarrow{\, m \,}
\ILtikzpic{
	\draw[thick,<-] (0,1) to[out=-45,in=-135] (1,1);
	\draw[thick,<-] (1,0) to[out=135,in=45] (0,0);
}
\qquad \text{or} \qquad
\ILtikzpic{
	\draw[thick,<-] (0,1) to[out=-45,in=45] (0,0);
	\draw[thick,<-] (1,0) to[out=135,in=-135] (1,1);
}
\xrightarrow{\, \Delta \,}
\ILtikzpic{
	\draw[thick,<-] (0,1) to[out=-45,in=-135] (1,1);
	\draw[thick,<-] (1,0) to[out=135,in=45] (0,0);
	\node at (.1,.4) {2};
	\node at (.9,.6) {1};
}
\]
\caption{Changing the local orientation at a crossing $c$ affects both rules (P) and (C) for every component $\del^e$ with edge $e$ corresponding to $c$.  One case is illustrated here, where the ordering of the circles changes for multiplication (but not for comultiplication), and all notions of local consistency swap.}
\label{fig:local orientation change}
\end{figure}

For the edges $e$ corresponding to $m$, we swap all notions of local consistency, which one may check has no effect on any signs in $\del^e$ via rule (C) (note that any non-zero multiplication will have an even number of $X$'s involved in the domain and codomain).  However we also swap the order of the circles in the domain, and thus $\del^e$ is negated via rule (P).

For the edges $e$ corresponding to $\Delta$, we again swap all notions of local consistency, which negates $\del^e$ via rule (C) (note that all non-zero comultiplications involve an odd number of $X$'s).  The ordering of the circles in the codomain is maintained, and so rule (P) is not affected.

Thus we see that the local orientation change in Figure \ref{fig:local orientation change} negates the full sum $\del^c$.  The permutative nature of the change ensures that swapping the understrand's orientation again effectively negates $\del^c$ once more (returning it to how it was).  We leave it to the reader to check that swapping the overstrand's orientation does not affect $\del^c$.  The original claim, that $\del^c$ is negated if and only if the understrand is changed, then follows from the permutative nature of the sign changes.

In general then if we make an arbitrary change of local orientations at various crossings, we let $\underline{C}$ denote the set of crossings for which the understrand has changed.  We can define our isomorphism $\phi$ on any $C(D_u)$ to be multiplication by $(-1)^{\sum_{c\in\underline{C}} u_c}$, where $u_c$ is the $c$-th entry of the vector $u\in\cube^n$.  It is a simple exercise then to verify that $\phi$ commutes with the differential $\del$ and that the naturality condition of Equation \eqref{eq:naturality of local ori change} holds.
\end{proof}

\begin{theorem}\label{thm:d squared is zero}
For any link diagram $D\subset \RP^2$, $(\unC(D),\del)$ is a chain complex (i.e. we have $\del^2=0$).
\end{theorem}
\begin{proof}
The proof is a case-by-case check over all 2-crossing link diagrams in $\RP^2$ just as in \cite[Lemmas 6.4,6.5,6.6]{Gabrovsek}.  In that paper it is claimed that \cite[Lemma 6.2]{Gabrovsek} about global link orientations allows one to choose any global orientation for such links (which correspond to 2-dimensional subcubes of the large cube of resolutions for $D$), but due to the fact that a global orientation for $D$ may not descend to a well-defined orientation for these subcube links, one in fact requires the stronger Lemma \ref{lem:change local ori gives chain iso} to make this reduction.

In addition, there is one 4-valent, 2-vertex graph in $\RP^2$ missed in \cite{Gabrovsek} shown in Figure \ref{fig:st=u 4-valent graph}. This is exactly the case that leads to the requirement that the coefficient of the map $\Phi^e_0$ in \eqref{eq:dele} is the product of the coefficients for $\del^e_-$ and $\Phi^e_+$. Indeed, for example, when both crossings in Figure \ref{fig:st=u 4-valent graph} are taken to be positive, we can take the zero resolution at each crossing to arrive at the resolution diagram $D_0$ consisting of two trivial circles and one essential circle.  One can then compute that $$\del^2(\ox \otimes X\otimes X) = \del(\pm t \oone \otimes X \pm st \ox \otimes 1) = \pm(st-st)\ox = 0,$$ indicating the requirement on coefficients.
\end{proof}

\begin{figure}
\[\begin{tikzpicture}[xscale=.3, yscale=.3]
\draw[thick] plot [smooth] coordinates { (0,-3) (0,-1) (1,1) (2,-1) (2.5,0) (2,1) (1,-1) (0,1) (0,3) };
\draw (0,0) circle (3);
\end{tikzpicture}\]
\caption{A 4-valent, 2-vertex graph in $\RP^2$ which is missing from Figure 10 in \cite{Gabrovsek}.}
\label{fig:st=u 4-valent graph}
\end{figure}
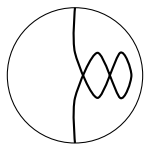

\begin{remark}\label{rmk:w=st in annular}
	In the annular theory of \cite{GLW} one may define analogous deformation parameters $s,t$ for the various graded components of the deformed differential.  In that setting the annular closure of the graph in Figure \ref{fig:st=u 4-valent graph} would lead to the same requirement that  the coefficient of $\Phi^e_0$ in \eqref{eq:dele} is the product of the coefficients for $\del^e_-$ and $\Phi^e_+$.  This is reflected in both theories by noting that $\deg(\del_-^e\circ\Phi_+^e) = \deg(\del_0^e\circ\Phi_0^e)$.
\end{remark}

The complex $(\unC(D),\del)$ will be referred to as the \emph{unnormalized deformed Khovanov complex} for $D$.  As is typical of Khovanov homology theories, an orientation for $D$ allows us to renormalize our gradings by defining
\[C(D):=\unC(D)[-n^-,n^+-2n^-,0],\]
where $[a,b,c]$ indicates a global shift in the tri-grading $(i,j,k)$ and $n^\pm$ is the number of positive/negative crossings in $D$.  We will also use the notation
\[C(D_v):=\unC(D_v)[-n^-,n^+-2n^-,0]\]
for the various normalized chain groups at each vertex in $C(D)$.

\begin{remark}\label{rmk:grading convention}
	Our grading conventions here differ slightly from the conventions in \cite{Gabrovsek}.  In particular, the grading denoted $j$ in \cite{Gabrovsek} is equivalent to $j-k$ in our notation, and our quantum grading $j$ is normalized in $C(D)$ to remove the dependency on the framing of the link.  Our conventions then align more closely with the conventions for annular Khovanov-Lee homology as in \cite{GLW}.
\end{remark}

\begin{remark}
	In \cite{Drobo}, Drobotukhina defined an analogue of the Jones polynomial for oriented links in $\RP^3$. After switching variables from $A$ to $q=-A^{-2}$, let us denote her invariant by $J_L(q)$. Then, it is not hard to check that the Euler characteristic of $C(D)$ is given by
	\begin{equation}
		\label{eq:Drobo}
		\sum_{i, j, k \in \Z} (-1)^i q^j x^k \operatorname{rk } C^{i, j, k}(D) = \begin{cases}
			(q+q^{-1}) \cdot J_{L}(q) & \text{if $L$ is of class-0},\\
			(qx + q^{-1}x^{-1})  \cdot J_{L}(q) & \text{if $L$ is of class-1}.
		\end{cases}
	\end{equation}
	By contrast, with the conventions in \cite{APS, Gabrovsek}, the Euler characteristic of their homology is the element $[L]$ in the Kauffman bracket skein module of $\RP^3$ (the analogue of the usual Kauffman bracket in $S^3$), which depends on the framing of $L$. 
\end{remark}

\begin{definition}
	Given an oriented link diagram $D\subset \RP^2$ and a set of auxiliary diagram choices, the \emph{deformed Khovanov complex} $ \KC_d^*(D)$ is the chain complex $(C(D),\del)$ where we choose $R =\Z[s, t]$, with $s$ and $t$ in \eqref{eq:dele} being the two polynomial variables.
\end{definition}

Observe that $\KC_d^*(D)$ contains enough information to recover all the other complexes $(C(D),\del)$, for any choice of commutative ring $R$ and elements $s, t \in R$. Indeed, this can be done by tensoring with $R$ over $\Z[s, t]$, where we view $R$ as a $\Z[s, t]$-module using the action of the elements $s$ and $t$.

In particular, we have the following two complexes.

\begin{definition}
The \emph{Khovanov complex} $ \KC^*(D)$ (over a commutative ring $R$) is the complex $(C(D),\del)$ where we have chosen $s=t=0$.  This is equivalent to the Khovanov complex defined in \cite{Gabrovsek}, but with alternative grading conventions as in Remark \ref{rmk:grading convention}. (Observe that $\del^e_0$ is the only part of the differential in  \eqref{eq:dele} that contributes to the Khovanov complex.)  The \emph{Khovanov homology} is then $\Kh^*(D)=H^*(\KC(D))$.
\end{definition}

\begin{definition}
The \emph{Lee complex} $LC^*(D)$ (over a commutative ring $R$) is the complex $(C(D),\del)$ where we have chosen  $s=t=1$.  The \emph{Lee homology} is then $LH^*(D)=H^*(LC(D))$.
\end{definition}

We have not included the auxiliary diagram choices of Definition \ref{def:auxiliary diagram choices} in our notation for the complexes above.  This is justified by the following theorem, which is a direct generalization of Lemma \ref{lem:change local ori gives chain iso} and \cite[Lemmas 6.1, 6.3]{Gabrovsek}.

\begin{theorem}\label{thm:  KC is well-defined}
Fix an oriented link diagram $D\subset\RP^2$.  Let $ \KC_d^*(D)$ and $ \KC_d^*(D)'$ denote the deformed complexes assigned to $D$ with two different sets of auxiliary diagram choices.  Then there is a natural  chain isomorphism
\[ \KC_d^*(D) \xrightarrow{\phi}  \KC_d^*(D)'\]
such that, if $ \KC_d^*(D)''$ is the complex using a third set of auxiliary diagram choices, the following diagram of natural isomorphisms commutes.
\begin{equation}\label{eq:composing natural isos for different auxiliary diagram choices}
\begin{tikzcd}
	 \KC^*_d(D) \ar[r,"\phi"] \ar[rr,bend left,"\phi"] &  \KC^*_d(D)' \ar[r,"\phi"] &  \KC^*_d(D)''
\end{tikzcd}
\end{equation}
\end{theorem}
\begin{proof}
The map $\phi$ is determined on each resolution as in Lemma \ref{lem:change local ori gives chain iso} and \cite[Lemmas 6.1, 6.3]{Gabrovsek}.  That is to say, if $g$ denotes a labelling of circles corresponding to a generator in $C(D_u)$, then $\phi(g)= \pm g$ with the same labels in the corresponding $C(D_u)'$.  The sign is determined by the changes of local orientations of crossings, and of circle orientations and orderings in $D_u$ as follows.

Suppose $D_u$ has $r$ circles.  Let $\sigma\in \Sigma_r$ denote the permutation of the circles in $D_u$ relating the ordering with the two auxiliary diagram choices.  Then for each $i=1,\dots,r$, let $\epsilon_i = 1$ if the orientation of the $i^{\text{th}}$ circle with respect to the auxiliary diagram choices used to define $C(D_u)$ matches the orientation of the $\sigma(i)^{\text{th}}$ circle with respect to the auxiliary diagram choices used to define $C(D_u)'$, and $\epsilon_i = -1$ otherwise.  Finally, let $X_u(g) \subset \{1,\dots,r\}$ denote the set of circles in $C(D_u)$ which are labeled $X$ (or $\ox$) by the generator $g$.  Then letting $\underline{C}$ denote the set of crossings where the understrand has changed (see the proof of Lemma \ref{lem:change local ori gives chain iso}), the map $\phi$ is given by the formula
\begin{equation}\label{eq:natural iso formula}
	\phi(g) := \Bigl( (-1)^{\sum_{c\in\underline{C}}u_c} \sign(\sigma) \prod_{i\in X_u(g)} \epsilon_i \Bigr) g.
\end{equation}
In words, any entry of $u\in\cube^n$ which is a 1 contributes a minus sign if the corresponding crossing had its understrand's local orientation changed.  Any circle labeled by $X$ in $g$ which changes orientation in $D_u$ also contributes a minus sign.  Finally, if the change in ordering of circles was given by an odd permutation, we include an extra minus sign.  Note that this extra sign is inherent in \Gabrov's own argument \cite{Gabrovsek} where $C(D_u)$ is defined as an exterior product rather than a tensor product.

The proof that $\phi$ commutes with the differential is similar to that in Lemma \ref{lem:change local ori gives chain iso} and \cite[Lemmas 6.1, 6.3]{Gabrovsek}.  The naturality with respect to composition given by diagram \eqref{eq:composing natural isos for different auxiliary diagram choices} follows from the formula of Equation \eqref{eq:natural iso formula} since the sign map respects composition of permutations.
\end{proof}

Finally, we turn to invariance under Reidemeister moves to show that the deformed Khovanov complex is in fact a link invariant.  For links in $\RP^3$ there are five Reidemeister moves, illustrated in \cite[Figure 3]{Drobo}.  These include the three classical Reidemeister moves needed for links in $S^3$, together with two additional moves R-IV and R-V which allow crossings and turnbacks to `pass through the boundary' of our $\RP^2$ of projection as shown in Figure \ref{fig:Reid 4 and 5}.

\begin{figure}
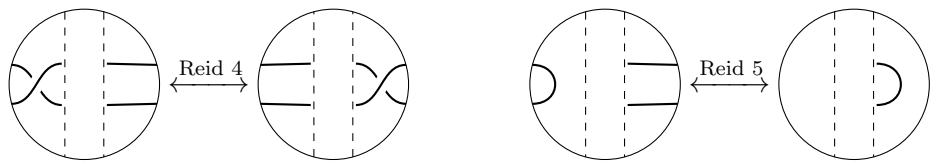

\[\ILtikzpic[scale=.5]{
\draw[thick] (165:2) to[out=0,in=180] (-.6,-.55);
\drawover[thick]{(195:2) to[out=0,in=180] (-.6,.55);}
\draw[thick] (15:2) -- (.6,.55);
\draw[thick] (-15:2)--(.6,-.55);
\draw (0,0) circle (2);
\draw[dashed] (105:2) -- (255:2);
\draw[dashed] (75:2) -- (285:2);
}
\xleftrightarrow{\text{Reid 4}}
\ILtikzpic[scale=-.5]{
	\draw[thick] (165:2) to[out=0,in=180] (-.6,-.55);
	\drawover[thick]{(195:2) to[out=0,in=180] (-.6,.55);}
	\draw[thick] (15:2) -- (.6,.55);
	\draw[thick] (-15:2)--(.6,-.55);
	\draw (0,0) circle (2);
	\draw[dashed] (105:2) -- (255:2);
	\draw[dashed] (75:2) -- (285:2);
}
\qquad\qquad
\ILtikzpic[scale=.5]{
	\draw[thick] (165:2) to[out=0,in=0,looseness=2] (195:2);
	\draw[thick] (15:2) -- (.6,.55);
	\draw[thick] (-15:2)--(.6,-.55);
	\draw (0,0) circle (2);
	\draw[dashed] (105:2) -- (255:2);
	\draw[dashed] (75:2) -- (285:2);
}
\xleftrightarrow{\text{Reid 5}}
\ILtikzpic[scale=.5]{
	\draw[thick] (.6,.55) to[out=0,in=0,looseness=2] (.6,-.55);
	\draw (0,0) circle (2);
	\draw[dashed] (105:2) -- (255:2);
	\draw[dashed] (75:2) -- (285:2);
}
\]
\caption{Reidemeister 4 and 5 moves for link diagrams in $\RP^2$.}
\label{fig:Reid 4 and 5}
\end{figure}

\begin{theorem}\label{thm:Reidemeister moves give chain homotopy equivalences}
Let $D'\subset \RP^2$ be a link diagram obtained from $D\subset \RP^2$ by performing one of the five Reidemeister moves.  Then there is a chain homotopy equivalence of $\Z[s, t]$-modules
\begin{equation}
\label{eq:phir}
 \KC_d^*(D) \xrightarrow{\phi_R}  \KC_d^*(D').
 \end{equation}
In particular, for any link $L\subset\RP^3$, the homology groups $ \Kh_d^*(L)$ are well-defined link invariants (as are the Khovanov homology $\Kh^*(L)$ and Lee homology $LH^*(L)$).
\end{theorem}
\begin{proof}
As in \cite[Section 6]{Gabrovsek}, the proofs of the the classical Reidemeister moves II and III follow along precisely the same reasoning as they did for links in $S^3$, in \cite{Kh, Lee}.  That is to say, in all three cases, one is able to identify acyclic subcomplexes and quotient complexes which can be removed up to chain homotopy.  These acyclic complexes arise from fixing single labels ($1$ or $X$) on trivial circles that appear in certain resolutions---the fact that no essential circles are used in identifying these complexes ensures that the arguments are equivalent whether working in $S^3$ or $\RP^3$.

In \cite{Gabrovsek}, invariance under Reidemeister I is not discussed, since the homology there is presented as an invariant of framed links.  However, the argument for Reidemeister I in \cite{Kh, Lee} translates from $S^3$ to $\RP^3$ in precisely the same manner as the other two moves, since the complex for a diagram with a kink has a disjoint trivial circle which is, once again, providing an acyclic subcomplex to collapse.

Meanwhile, the fourth and fifth Reidmeister moves trivially induce chain isomorphisms since there is a natural isotopy in $\RP^2$ between the various resolutions of the diagrams before and after these moves.  Note that such an isotopy allows any set of auxiliary diagram choices before the move to determine a set of auxiliary diagram choices for after the move, to ensure that the corresponding map is indeed a chain map (i.e., that it respects the signs in the differential).
\end{proof}

As with all diagrammatic link homology theories coming from Khovanov cubes of resolutions, there are skein exact sequences associated to crossings in a diagram.

\begin{proposition}\label{prop:skein exact seq}
Fix an oriented link diagram $D\subset\RP^2$.  For any fixed crossing $c$ of $D$, let $D_0,D_1$ denote the link diagrams arising from taking the zero- and one-resolutions of $D$ at $c$, oriented in any fixed manner.  Then for any sets of auxiliary diagram choices for $D,D_0,D_1$ we have a short exact sequence of chain complexes (ignoring grading shifts)
\[0\rightarrow \KC_d(D_1) \rightarrow \KC_d(D) \rightarrow \KC_d(D_0) \rightarrow 0,\]
which induces a long exact sequence on homology
\[\cdots \rightarrow \Kh_d(D_1)\rightarrow \Kh_d(D) \rightarrow \Kh_d(D_0) \rightarrow \Kh_d(D_1) \rightarrow \cdots\]
The grading shifts omitted here are determined by the orientations chosen for $D,D_0,D_1$.
\end{proposition}
\begin{proof}
Given a set of auxiliary diagram choices for $D$, it is clear from the definition of the unnormalized complex that $\unC(D)\cong\mathrm{Cone}(\unC(D_0)\xrightarrow{\del^c} \unC(D_1)),$ where the auxiliary diagram choices for each $\unC(D_i)$ are inherited from the choices for $D$. This gives rise to a short exact sequence
\[0\rightarrow \unC(D_1) \rightarrow \unC(D) \rightarrow \unC(D_0) \rightarrow 0\]
for these fixed auxiliary diagram choices.  The proposition then follows from Theorem \ref{thm:  KC is well-defined} together with the fact that $\KC_d(D)$ and $\KC_d(D_i)$ are simply renormalizations of the complexes $\unC(D)$ and $\unC(D_i)$, with shifts determined by orientations of the link diagrams.
\end{proof}

We end this section with some basic properties of the complexes inherited from Lemma \ref{lem:ess circ in Dv iff L is class-1}.

\begin{proposition}\label{prop:symmetry for class 1}
Let $D\subset\RP^2$ be a link diagram for a link $L\subset\RP^3$.
\begin{itemize}
\item If $L$ is class-1, then every generator in $ \KC^*_d(D)$ has $k$-grading $\pm 1$.  
\item If $L$ is class-0, then every generator in $ \KC^*_d(D)$ has $k$-grading zero. 
\item Consider the involutions $ \KC^*(D)\xrightarrow{\Theta}  \KC^*(D)$ and $LC^*(D)\xrightarrow{\Theta} LC^*(D)$ (compare \cite[Lemma 4]{GLW}) which interchange the labels $\oone$ and $\ox$ on the essential circles in the resolutions $D_v$ of $D$ (while maintaining all other labels). Then $\Theta$ are chain maps, producing involutions between the respective homologies. When $L$ is class-0, these involutions are the identity, whereas when $L$ is class-1, they produce isomorphisms between the homology in $k$-grading $+1$ and that in $k$-grading $-1$.
\end{itemize}
\end{proposition}

\begin{proof}
Lemma \ref{lem:ess circ in Dv iff L is class-1} implies the $k$-grading statements immediately.  If $s=t$, Table \eqref{eq:differential table} can then be used to verify that $\Theta$ is a chain map, just as in \cite{GLW,GW:AnnLinks}. We have $s=t$ for both the Khovanov complex and the Lee complex. When $L$ is class-0, there are no essential circles in any resolutions, so $\Theta$ is clearly the identity. When $L$ is class-1, the map $\Theta$ interchanges the $k$-gradings $1$ and $-1$; being an involution, it must be an isomorphism.
\end{proof}

\begin{remark} \label{rem:bigraded}
While the Khovanov homology $\Kh^*(L)$ is a priori triply graded, Proposition~\ref{prop:symmetry for class 1} makes it clear that the last grading $k$ is not essential, because all the information is contained in the bigraded piece $\Kh^{*, *, 0}$ when $L$ is class-0 and in $\Kh^{*,*,1}$ when $L$ is class-1. This is similar to how the Euler characteristic $\chi(\Kh^*(L))$ is determined by the polynomial $J_L(q)$ in a single variable; see \eqref{eq:Drobo}.
\end{remark}

\subsection{Simplifying the signs}
The signs involved in the differential $\del$ for $\KC^*_d(D)$ are determined by the set of auxiliary diagram choices for $D$.  In this section we seek to simplify these choices.  We note first that an orientation for $D$ induces a natural choice of local orientation at each crossing (this is the choice being implicitly made in \cite{Gabrovsek}).  From this point forward we assume that, for oriented $D$, we are making this choice of local orientations unless stated otherwise.

From here we are left with the choices of circle orderings and orientations in the various resolutions $D_v$.  However, not all of these choices affect the signs in the same way.  Note that, in general, the sign rule (P) indicates that a choice of circle ordering affects an entire edge map $\del^e$ at a time, while rules (O) and (C) indicate that a choice of orientations may affect different generators differently.  Thus it is the choice of orientations for the various $D_v$ which contributes to the main complication in the signs.  Fortunately there are certain choices of convenient orientations for a given $D_v$ which can be packaged via a single choice of curve in $\RP^2$.

\begin{definition}
\label{def:dividing circle}
A \emph{dividing circle} for a resolution diagram $D_v$ is an oriented essential circle $\mc{C}$ in $\RP^2$ such that 
\[
\mc{C}\cap D_v = \begin{cases}
		\mc{C} & \text{if $D_v$ contains an essential circle,}\\
		\varnothing & \text{otherwise.}
	\end{cases}
\]
In either case, a dividing circle $\mc{C}$ for $D_v$ induces an orientation $o_{\mc{C}}$ on $D_v$ by orienting each circle in alternating fashion according to its distance from $\mc{C}$.  Specifically, the complement of $\mc{C}$ in $\RP^2$ is a disk, which we orient so that its boundary orientation coincides with that of $\mc{C}$; that is, if we identify the disk with the standard $B^2$ preserving orientations, the circle $\mc{C}$ should be oriented counterclockwise. Given a circle $\mc{C}'\neq \mc{C}$ in $D_v$, we let its distance from $\mc{C}$ be the minimal number of other circles in $D_v$ that a path from $\mc{C}$ to $\mc{C}'$ needs to intersect, plus $1$. Then, we orient $\mc{C}'$ counterclockwise (via the identification with the standard disc $B^2$) if and only if this distance is even.  See Figure \ref{fig:Div circ example}.\end{definition}

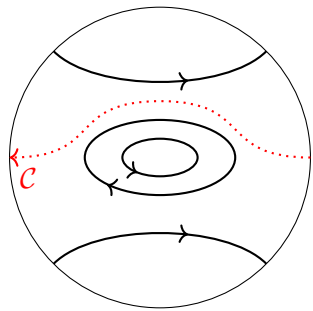
\begin{figure}
	\[\begin{tikzpicture}[scale=.5]
		\draw (0,0) circle (4);
		\draw[thick,
		decoration={markings, mark=at position 0.625 with {\arrow{>}}}, postaction={decorate}]
		(0,0) ellipse (1 and .5);
		\draw[thick,
		decoration={markings, mark=at position 0.625 with {\arrow{<}}}, postaction={decorate}]
		(0,0) ellipse (2 and 1);
		\draw[thick,
		decoration={markings, mark=at position 0.625 with {\arrow{>}}}, postaction=decorate]
		(135:4) to[out=-45,in=-135,looseness=.7] (45:4);
		\draw[thick,
		decoration={markings, mark=at position 0.625 with {\arrow{>}}}, postaction={decorate}]
		(-135:4) to[out=45,in=135,looseness=.7] (-45:4);
		\draw[thick,dotted,red,->]
		(4,0) to[out=180,in=0,looseness=1.5] (0,1.5) to[out=180,in=0,looseness=1.5] (-4,0) node[below right]{$\mc{C}$};
	\end{tikzpicture}\]
\caption{For a resolution diagram $D_v\subset \RP^2$, a choice of oriented dividing circle $\mc{C}$ (illustrated here with a dotted red curve) induces an orientation $o_{\mc{C}}$ on $D_v$.  If $D_v$ contains an essential circle, we choose $\mc{C}$ to be this circle.  Note that $o_{\mc{C}}$ allows any saddle cobordism to or from $D_v$ which does not pass through $\mc{C}$ to be orientable (including the case when $\mc{C}$ is an essential circle in $D_v$ oriented accordingly).}
\label{fig:Div circ example}
\end{figure}

\begin{remark}
A similar choice of orientations for resolutions of planar diagrams was described in \cite[Section 2.4]{Rasmussen}, using the distance from infinity.
\end{remark}

The orientations $o_\mc{C}$ of Definition \ref{def:dividing circle} will be used throughout the paper to simplify various computations involving signs, and will also be used to define the generators of Lee homology.  Furthermore, there are some cases where these orientations can be used to eliminate all of the sign complications coming from rules (O) and (C).  The key to realizing this is the following lemma about terms in the differential corresponding to oriented saddles.

\begin{lemma}\label{lem:signs for oriented saddle}
Let $u\xrightarrow{e}v$ be an edge in the cube $\cube^n$, with corresponding differential component $C(D_u)\xrightarrow{\del^e}C(D_v)$ between resolutions of the link diagram $D\subset\RP^2$.  Suppose orientations $o_u,o_v$ have been chosen for $D_u,D_v$ respectively such that the saddle cobordism $D_u\xrightarrow{s}D_v$ is orientable with respect to $o_u,o_v$.  Then the map $\del^e$ sends all generators $g\in C(D_u)$ to sums of generators in $C(D_v)$ of the same sign.  In other words, rules (O) and (C) affect the entire edge map $\del^e$ equivalently, just as rule (P) does.
\end{lemma}
\begin{proof}
For an orientable saddle $s$, orientations on distant circles are maintained, so that rule (O) cannot contribute anything.  Meanwhile, for rule (C), one sees that the consistency of each nearby circle is maintained from domain to codomain, and for both $m$ and $\Delta$ the parity of the $X$-counts is constant amongst all generators (see also the similar argument in the proof of Lemma \ref{lem:change local ori gives chain iso}).  This indicates that all generators are affected equally by rule (C).
\end{proof}

\begin{definition}[Definition 4.5 in \cite{LS}]
A {\em sign assignment} for the cube $\cube^n$ consists in signs $\pm 1$ assigned to each edge of the cube, such that the product of the signs along the edges of any $2$-dimensional face is $-1$.
\end{definition}

With the help of Lemma \ref{lem:signs for oriented saddle}, we can simplify the signs on all edges simultaneously for both local links and class-1 links.

\begin{theorem}\label{thm:local link is same as S3}
If $D\subset\RP^2$ is a local link diagram for the (necessarily local) link $L$, then the complex $LC^*(D)$ (respectively $ \KC^*(D)$) is chain isomorphic to the usual Lee complex (respectively Khovanov complex) for $D$ viewed as a diagram in $B^2$ for a link in $S^3$.  
\end{theorem}
\begin{proof}
By Theorem \ref{thm:  KC is well-defined}, we may choose any set of orientations on the various $D_u$ that we like.  We will do so via dividing circles as in Definition \ref{def:dividing circle}.  Since $D$ is a local link diagram, we may choose a single oriented essential circle $\mc{C}\subset\RP^2\setminus D$ which acts as a dividing circle for all the various resolutions $D_u$.  This circle $\mc{C}$ then induces orientations $o_{\mc{C}}$ on each $D_u$ in such a way that all saddles in the cube of resolutions are oriented.  Lemma \ref{lem:signs for oriented saddle} then implies that all signs are determined edge-wise only, so that the cube must be a commuting cube with an edge sign assignment, and it is known that all such sign assignments give isomorphic complexes; see for example \cite[Lemma 5.7 and Proposition 5.8]{Putyra}.
\end{proof}

\begin{theorem}\label{thm:essential circle in resolutions gives sign assignment}
If $D$ is the diagram of a class-1 link $L$ in $\RP^3$, then the complexes $ \KC^*_d(D)$, $LC^*(D)$, $ \KC^*(D)$ are each isomorphic to corresponding complexes determined by the unsigned Table~\eqref{eq:differential table},  together with a choice of sign assignment for the cube $\cube^n$.
\end{theorem}
\begin{proof}
Again Theorem \ref{thm:  KC is well-defined} allows us to choose our orientations.  Since $L$ is class-1, Lemma \ref{lem:ess circ in Dv iff L is class-1} shows that every resolution $D_v$ must contain a (unique)  essential circle and we can choose these essential circles to be our dividing circles $\mc{C}$. We also choose an orientation for the essential circle in the oriented resolution of $D$. This determines orientations for all the circles in the oriented resolution, in such a way that all the saddles connecting it to  another resolution are oriented. From here we get induced orientations on each $D_v$ in such a way that all saddles are oriented, so that Lemma \ref{lem:signs for oriented saddle} completes the proof.
\end{proof}

\begin{remark}
The two theorems above correspond to the two cases where there are no 1-1 bifurcations (see also Remark \ref{rmk:No 1-1 bifurcations for class-1}).  If there are 1-1 bifurcations, then it is impossible for the cube to commute without signs except in some degenerate cases, and we cannot expect anything similar to the above theorems.
\end{remark}

\section{Lee generators and Lee homology} \label{sec:Lee}
The goal of this section is to prove Theorem \ref{thm:Lee} computing the Lee homology $LH^*(D)$ of any link diagram $D\subset\RP^2$. We will use arguments analogous to those of Lee in \cite{Lee}.  As such, we begin with a change of basis which diagonalizes the differential $\del_L$ and indicates the existence of an adjoint $\del_L^*$.  In order to assign Lee generators in $LC^*(D)$ which generate $LH^*(D)$ however, it will be necessary to utilize certain preferred sets of auxiliary diagram choices, and to keep track of the naturality of our generators with respect to these choices.

From now on we will assume that the base ring $R$ is $\Q$. (More generally, everything will still go through if we just assume that $2$ is invertible in $R$.)

\subsection{The Lee basis}

Let $D\subset\RP^2$ be any link diagram, and fix an arbitrary set of auxiliary diagram choices for $D$.  In Section \ref{sec:Defining deformed cx and Lee cx} we described the chain groups $C(D_v)$ involved in $LC^*(D)$ as tensor products of $V=\langle 1,X \rangle$ and $\ov=\langle \oone,\ox \rangle$.  As in \cite{Lee}, we define a new basis
\[ a:= 1+X,\quad b:= 1-X,\quad \oa:=\oone+\ox,\quad \ob:=\oone-\ox,\]
so that $V=\langle a,b \rangle$ and $\ov=\langle \oa,\ob \rangle$.

Unlike in $S^3$, here the differential on (tensor products of) such elements still requires some care based upon the local orientations of the circles near the saddle.  If the saddle is a 1-1 bifurcation, then the differential is zero and no further analysis is needed.  Otherwise, we present the following tables for $2\rightarrow 1$ bifurcations (multiplication $m$) and $1\rightarrow 2$ bifurcations (comultiplication $\Delta$) which utilize the notation $C$ and $I$ for consistently and inconsistently oriented circles near the saddle (see Rule (C) in Section \ref{sec:Defining deformed cx and Lee cx}).

\begin{center}
\setlength{\extrarowheight}{2pt}{
\begin{tabular}{||C||C|C|C|C||C||C|C||}
\hline
m: & a\otimes a & a\otimes b & b\otimes a & b\otimes b & \Delta: & a & b \\
\hline
C\otimes C \rightarrow C & 2a & 0 & 0 & 2b & C\rightarrow C\otimes C & a\otimes a & -b\otimes b\\
C\otimes C \rightarrow I & 2b & 0 & 0 & 2a & I\rightarrow C\otimes C & -b\otimes b & a\otimes a\\
C\otimes I \rightarrow C & 0 & 2a & 2b & 0 & C\rightarrow C\otimes I & a\otimes b & -b\otimes a\\
C\otimes I \rightarrow I & 0 & 2b & 2a & 0 & I\rightarrow C\otimes I & -b\otimes a & a\otimes b\\
\hline

\end{tabular}
}
\end{center}

We list only four of the eight possible rows because the other four rows can be determined by the global involution which interchanges the meaning of ``consistent" with ``inconsistent".  Just as in the proof of Lemma \ref{lem:change local ori gives chain iso}, by counting the total number of circles which can be labeled $X$ in both domain and co-domain, one can check that, for the multiplication $m$, this involution is in fact the identity (and thus e.g. the row for $I\otimes I \rightarrow I$ is precisely the same as for $C\otimes C \rightarrow C$), whereas for the comultiplication $\Delta$, this involution is multiplication by $-1$ (and thus e.g. the row for $I \rightarrow I\otimes I$ is the same as for $C \rightarrow C\otimes C$ but with all entries negated).

This table also continues to hold if there is an essential circle involved, with $\oa$ and $\ob$ replacing $a$ and $b$ for that circle in both domain and codomain.

From this diagonalization it is clear that an adjoint differential $\del_L^*$ can be defined by swapping the role of $m$ and $\Delta$ (and including a factor of $\pm 2$ as necessary).

\subsection{Lee generators}
For the case of links in $S^3$, each orientation of the given link diagram determines a Lee generator which generates a summand of Lee homology \cite{Lee}.  For links in $\RP^3$ however, the complex $LC^*(D)$ is only well-defined up to a set of auxiliary diagram choices (see Definition \ref{def:auxiliary diagram choices} and Theorem \ref{thm:  KC is well-defined}).  As such, our assignments of Lee generators will depend upon these choices, and we will restrict ourselves to certain `preferred' choices for one resolution in particular.

To begin, we let $D\subset\RP^2$ denote a link diagram with orientation $o$.  Throughout this section we continue to assume that all auxiliary diagram choices mentioned are using $o$ to determine the choice of local orientations at crossings.  We then let $D_o$ denote the oriented resolution of $D$ determined by $o$. (Note that, for the opposite orientation $\oo$, we have $D_{\oo}=D_o$.)  The orientations $o,\oo$ for $D$ induce orientations on $D_o=D_{\oo}$.  Our definition of a Lee generator will depend upon a choice of dividing circle $\mc{C}$ for $D_o$ which in turn induces a third orientation $o_\mc{C}$ on $D_o$.

\begin{definition}\label{def:Lee generators}
Let $(D,o)$ be an oriented link diagram in $\RP^2$.  Fix a choice of dividing circle $\mc{C}$ for $D_o=D_{\oo}$.  Let $LC^*_{\mc{C}}(D)$ denote the Lee complex for $D$ using any fixed set of auxiliary diagram choices which orients the circles in $D_o=D_{\oo}$ by $o_{\mc{C}}$ as described in Definition \ref{def:dividing circle}.  Then the \emph{Lee generators for $(D,o)$ with respect to $\mc{C}$} are two elements $\LeegenC_o,\LeegenC_{\oo}\in C(D_o)=C(D_{\oo})$ defined as follows.  For $\LeegenC_o$ (respectively $\LeegenC_{\oo}$), on each circle of $D_o$ we compare the orientation induced by $o$ (respectively $\oo$) with the orientation $o_{\mc{C}}$.  If these orientations match, we label the circle by $a$ (or by $\oa$ if we are labeling the essential circle).  Otherwise, we label the circle by $b$ (or by $\ob$).
\end{definition}

The Lee generators for a given oriented link diagram are natural, up to a sign, with respect to the choice of dividing circle (as well as all of the other auxiliary diagram choices inherent in constructing the Lee complex for a fixed diagram).

\begin{proposition}\label{prop:Lee gens under change of div circ}
Fix an oriented link diagram $(D,o)\subset\RP^2$.  Let $LC^*_{\mc{C}}(D)$ (respectively $LC^*_{\mc{C}'}(D)$) denote the Lee complex assigned to $D$ with some set of auxiliary diagram choices which assigns the orientation $o_\mc{C}$ (respectively $o_{\mc{C}'}$) to $D_o$.  Then the natural isomorphism
\[LC^*_\mc{C}(D) \xrightarrow{\phi} LC^*_{\mc{C}'}(D)\]
of Theorem \ref{thm:  KC is well-defined} is a filtered map of degree zero satisfying
\[\phi(\LeegenC_o) = \pm\Leegen^{\mc{C}'}_o,\quad\quad \phi(\LeegenC_{\oo}) = \pm\Leegen^{\mc{C}'}_{\oo}.\]
\end{proposition}
\begin{proof}
The map is clearly filtered of degree zero since it only affects the sign of certain generators.

For a change in local orientations of crossings, $\phi$ is $\pm I$ on all resolutions (including on $C(D_o)$).  Similarly, for permutations of orderings of circles on $D_o$ (but $o_\mc{C}=o_{\mc{C}'}$), $\phi$ is $\pm I$ on $C(D_o)$ depending on the sign of the permutation.  Any change of circle orderings or orientations on resolutions other than $D_o$ causes $\phi$ to be $I$ on $C(D_o)$.

The interesting case is when $\mc{C}$ and $\mc{C}'$ are different enough to induce different orientations $o_\mc{C}\neq o_{\mc{C}'}$ on $D_o$.  In this case, the map $\phi\vert_{C(D_o)}$ maintains circles labeled $1$, but negates any circle labeled $X$ which changes orientation.  This is equivalent to interchanging $a$ with $b$ on any circle which changes orientation.  Since the original oriented diagram $(D,o)$ is unchanged, changing the orientation of a circle is the same as changing its comparison with $o$ (similarly with $\oo$), implying that $\phi$ sends $\LeegenC_o$ to $\Leegen^{\mc{C}'}_o$ (similarly with $\oo$) as desired.
\end{proof}

We postpone the naturality of the Lee generators with respect to the choice of diagram (i.e., invariance under Reidemeister moves as in Theorem \ref{thm:Reidemeister moves give chain homotopy equivalences}) until Section \ref{sec:Reid moves}.  In the meanwhile, following \cite{Lee}, we compute the Lee homology for any link diagram $D\subset\RP^2$.

\begin{proposition}\label{prop:Lee gens survive in homology}
Fix an oriented link diagram $(D,o)$ in $\RP^2$ and fix a choice of dividing circle $\mc{C}$ for resolution $D_o$.  Then for any set of auxiliary diagram choices which orients the circles in $D_o=D_{\oo}$ by $o_{\mc{C}}$, the Lee generators $\LeegenC_o,\LeegenC_{\oo}\in LC^*_{\mc{C}}(D)$ are both in $\ker(\del)\cap\ker(\del^*)$, and thus give distinct non-zero elements of $LH^*_{\mc{C}}(D)$.
\end{proposition}
\begin{proof}
Let us decorate the resolution $D_o$ with $n$ arcs indicating each of the saddles that can change the resolution at each of the $n$ crossings in $D$.  Let $D_o'$ denote the resolution diagram $D_o$ together with these arcs.

We focus on a single arc in $D_o'$, indicating a saddle $s$ that corresponds to a map $\del_s$ which contributes to either $\del$ if the original crossing in $D$ was positive (i.e. we took a 0-resolution here to arrive at $D_o$), or to $\del^*$ if the original crossing in $D$ was negative (i.e. we took a 1-resolution here to arrive at $D_o$).  In either case, we let $D_v$ denote the resolution diagram arrived at by performing the saddle cobordism, so that we have $C(D_o)\xrightarrow{\del_s}C(D_v)$.  We will use the choice of orientation $o_{\mc C}$ on $D_o$ (for the purposes of sign rules (O) and (C)) to show that $\del_s(\LeegenC_o)=\del_s(\LeegenC_{\oo})=0$ regardless of $s$ and $D_v$.

To begin, we note that the orientation $o$ on $D_o$ prevents any saddle $s$ from inducing a $1\rightarrow 2$ bifurcation.  If a saddle $s$ induces a $1\rightarrow 1$ bifurcation, then $\del_s$ is the zero map and there is nothing to prove.  Therefore we may assume for the remainder of the proof that $\del_s=m$, a multiplication map involving two circles merging into one.  From here there are two cases to consider.

Suppose that the original orientation $o$ matches $o_{\mc{C}}$ on precisely one of the two merging circles in $D_o$; up to interchanging the role of $o$ and $\oo$, we may assume that the orientations match on the eastern circle, where north is defined using the orientation $o$ of the crossing $c$ as in the rules (P) and (C) for determining signs of the differential.  Then $\del_s(\LeegenC_o)$ is computed by multiplying $b\otimes a$ with both circles consistent, while $\del_s(\LeegenC_{\oo})$ is computed by multiplying $a\otimes b$ with both circles inconsistent.  In such a case, the multiplication table for Lee generators shows $\del_s(\LeegenC_o)=\del_s(\LeegenC_{\oo})=0$ regardless of the orientation of $D_v$, as desired.

Now suppose that the original orientation $o$ matches $o_{\mc{C}}$ on an even number of the two merging circles in $D_o$; up to interchanging the role of $o$ and $\oo$, we may assume that $o$ and $o_{\mc{C}}$ match on both circles.  In this case $\del_s(\LeegenC_o)$ is computed by multiplying $a\otimes a$ with only the eastern circle consistent, while $\del_s(\LeegenC_{\oo})$ is computed by multiplying $b\otimes b$ with only the western circle consistent.  Once again the multiplication table for Lee generators shows $\del_s(\LeegenC_o)=\del_s(\LeegenC_{\oo})=0$ regardless of the orientation of $D_v$, as desired.
\end{proof}

\begin{theorem}\label{thm:Lee homology of diagram}
Let $D\subset\RP^2$ be a fixed link diagram for some link $L\subset\RP^3$ with $\ell$ components.  Let $O(L)$ denote the set of orientations of $L$.  For each pair $o,\oo \in O(L),$ choose a dividing circle $\mc{C}_o=\mc{C}_{\oo}$ for the oriented resolution $D_o=D_{\oo}$.  Then for any set of auxiliary diagram choices which assigns the orientation $o_{\mc{C}_o}$ on each $D_o$, the Lee homology will satisfy
\begin{equation}\label{eq:Lee homology of a diagram}
	\LH^*(D)\cong \Q^{O(L)},
\end{equation}
with each summand being generated by the corresponding Lee generator $\LeegenCi[o]_o$.
\end{theorem}
\begin{proof}
This is entirely similar to Lee's proof of Theorem 4.2 in \cite{Lee}.  The statement is proved for knots and 2-component links first, using an inductive argument on the number of crossings of the diagram $D$, and starting with the unknot as the base case.  For the inductive step, we consider the long exact sequence of Proposition \ref{prop:skein exact seq} (letting the auxiliary choices for $D$ determine the choices for $D_0$ and $D_1$) in the case of $s=t=1$ relating $\LH^*(D)$ to the Lee homology of the two resolutions $D_0$ and $D_1$ at a crossing.  
There are now three cases to consider depending on whether $L$ was a knot, or a split link, or a non-split link.

If $L$ was a knot, then one of $D_0$ and $D_1$ is also the diagram for a knot, while the other is for a 2-component link.  If $L$ was a non-split link, then the crossing in $D$ can be chosen so that $D_0$ and $D_1$ are both knots.  In either case, the inductive assumption applies to $\LH(D_0)$ and $\LH(D_1)$, and then one keeps track of what happens to the elements $\LeegenCi[o]_o$ under the exact sequence to arrive at the desired conclusion.  Meanwhile, if $L$ was a split link, we could choose a split diagram $D=D'\sqcup D''$ so that $\LH(D)\cong\LH(D')\otimes\LH(D'')$ and the result follows quickly.

From here, one provides a further induction on the number of components $\ell\geq 2$ using essentially the same arguments for split and non-split $L$.
\end{proof}

\section{The $s$-invariant for oriented links in $\RP^3$}
\label{sec:s-invt defined}

\subsection{The $s$-invariant of an oriented link diagram $D\subset\RP^2$}

Let $(D,o)\subset\RP^2$ be an oriented link diagram for a link $L\subset\RP^3$.  Theorem \ref{thm:Lee homology of diagram} states that, for any choice of dividing circle $\mc{C}$ for the oriented resolution $D_o$, the Lee generators $\LeegenC_o,\LeegenC_{\oo}$ generate two summands in the Lee homology $\LH^*(D)$.  As seen in Section \ref{sec:Defining deformed cx and Lee cx}, the differentials in the Lee complex change the quantum grading $j$ by either $0$ or $4$. Thus, the Lee complex is filtered with respect to $j$, and this filtration descends to the Lee homology.\footnote{The filtered subcomplexes here are given by $j \geq j_0$, for different values $j_0$. This convention is somewhat nonstandard (one usually considers $j \leq j_0$), but it is the same as the one in \cite{Rasmussen}.}  We will use the notation $\q(z)$ to denote the quantum filtration level of an element $z\in\LH^*(D)$.

\begin{definition}\label{def:s invariant for a diagram}
The \emph{Rasmussen $s$-invariant} of the oriented link diagram $(D,o)\subset\RP^2$ is defined to be
\[s(D):= \frac{\q\left(\left[\LeegenC_o + \LeegenC_{\oo}\right]\right) + \q\left(\left[\LeegenC_o - \LeegenC_{\oo}\right]\right)}{2},\]
where $\mc{C}$ is any choice of dividing circle for the oriented resolution $D_o$.
\end{definition}

We will see in Corollary \ref{cor:smin and smax} that, as is the case for links in $S^3$, the annulus, and $S^1\times S^2$, the two filtration levels $\q\left(\left[ \LeegenC_o + \LeegenC_{\oo} \right]\right)$ and $\q\left(\left[ \LeegenC_o - \LeegenC_{\oo} \right]\right)$ differ by two, and so we can let
\[\smin(D) := \q\left(\left[ \LeegenC_o \right]\right)\]
denote the lesser of these two quantities, so that $s(D)=\smin(D)+1$.

Proposition \ref{prop:Lee gens under change of div circ} shows that $s$ is well-defined for link diagrams (i.e. does not depend on the choice of dividing circle $\mc{C}$).  In order to show that $s$ is well-defined for links rather than link diagrams, we will need to analyze the behavior of Lee generators, and their filtration levels, under Reidemeister moves.

\subsection{Lee generators and filtration levels under Reidemeister moves}
\label{sec:Reid moves}
In this section we analyze the effect of Reidemeister moves on Lee homology.  Given such a Reidemeister move $R$ between two link diagrams, Theorem \ref{thm:Reidemeister moves give chain homotopy equivalences} provides a filtration-preserving chain homotopy equivalence $\phi_R$ on Lee complexes.  We wish to show that this map $\phi_R$ also preserves Lee generators in the proper sense.  The proof is diagrammatic and similar in spirit to the proofs in \cite{Rasmussen,BW}.  However, the complexes here depend not just upon the diagrams used for the links, but also on sets of auxiliary diagram choices for these diagrams.  As such, we will need to keep track of these choices when analyzing $\phi_R$.

Given an oriented link diagram $(D,o)\subset\RP^2$ and a choice of dividing circle $\mc{C}$ for the oriented resolution $D_o$, we will continue to use the notation $\LCC(D)$ for the Lee complex of $D$ using a set of auxiliary diagram choices which uses $o$ to assign local orientations to crossings and which orients $D_o$ using $o_{\mc{C}}$, together with the notation $\LeegenC_o,\LeegenC_{\oo}\in \LCC(D)$ for the corresponding Lee generators of Definition \ref{def:Lee generators}.

\begin{lemma}[Reidemeister 1]\label{lem:Reid 1 on Lee gens}
Let $D_1,D_2\subset \RP^2$ be two oriented link diagrams related by a single Reidemeister 1 move (either left- or right-handed).  Then there exists a choice of dividing circles $\mc{C}_i$ for $D_{i,o}$ such that the chain homotopy equivalence of Theorem \ref{thm:Reidemeister moves give chain homotopy equivalences} descends to a map on homology
\[\LHCi[1](D_1) \xrightarrow{\phi_R} \LHCi[2](D_2)\]
which satisfies
\[\phi_R([\LeegenCi[1]_o]) = \lambda [\LeegenCi[2]_o], \]
where $\lambda$ is some invertible scalar in $\Q$ (and likewise for the opposite orientation $\oo$).
\end{lemma}

\begin{proof}
We present the following argument with an eye towards generalization to the Reidemeister 2 and 3 moves to come, although in the current setting of the Reidemeister 1 move, various simplifications are possible.  We will focus on the right-handed Reidemeister 1 move; the left-handed analysis is similar.

The diagrams $D_1,D_2$ are identical outside of a small local picture; without loss of generality we may assume the two local pictures are
\[
D_1 = \ILtikzpic[xscale=.2,yscale=.2]{
				\draw[thick,->] (2.5,0) to[out=-90, in=0] (2,-1) to[out=180, in=-90] (0,2)--(0,3.5);
				\drawover[thick] { (0,-3.5)--(0,-2) to[out=90, in=180] (2,1) to[out=0, in=90] (2.5,0); }
				\draw[dashed] (0,0) circle (3.5);
			}
\quad , \quad
D_2 = \ILtikzpic[xscale=.2,yscale=.2]{
				\draw[thick,->] (0,-3.5)--(0,3.5) ;
				\draw[dashed] (0,0) circle (3.5);
			}\quad ,
\]
with oriented resolutions also identical outside of the local pictures
\[
D_{1,o} = \ILtikzpic[xscale=.2,yscale=.2]{
					\draw[thick, ->] (0,-3.5)--(0,3.5);
					\draw[thick] (1.8,0) circle (1);
					\draw[thick,->] (.8,.1)--(.8,.2);
					\draw[dashed] (0,0) circle (3.5);
					}
\quad , \quad
D_{2,o} =  \ILtikzpic[xscale=.2,yscale=.2]{
					\draw[thick,->] (0,-3.5)--(0,3.5) ;
					\draw[dashed] (0,0) circle (3.5);
					}\quad .
\]
Note that we have drawn the inherited orientations, although these orientations are \emph{not} the ones we use in our set of auxiliary diagram choices (recall that, after choosing a dividing circle $\mc{C}_i$, we use the alternating orientation $o_{\mc{C}_i}$ to make those choices).  Indeed these inherited orientations will play no role in the Reidemeister 1 setting, but will play a role when analyzing Reidemeister 2 and 3.

We now let $\Gamma$ denote the connected component within $D_{2,o}$ of the single strand present in the local picture above; thus $\Gamma$ is a circle in $\RP^2$.  There are now two cases to consider.  If the circle $\Gamma$ is local (contained in a disk in $\RP^2$), then the dividing circle $\mc{C}_2$ can be chosen entirely disjoint from this local picture.  In this case, since $D_{1,o}$ is identical to $D_{2,o}$ away from this local picture, we can choose $\mc{C}_1=\mc{C}_2$ disjointly as well.

On the other hand, if the circle $\Gamma$ is homologically essential, it must be chosen as $\mc{C}_2$.  Once again, since $D_{1,o}$ is identical to $D_{2,o}$ outside of this local picture, this same circle must be essential in $D_{1,o}$ and we can again choose $\mc{C}_1=\mc{C}_2$.

In both cases, since $\mc{C}_1=\mc{C}_2$, we are able to set all of the resolution choices for $D_1$ and $D_2$ equivalently (when ordering circles, the extra disjoint circle in various resolutions of $D_1$ can always be ordered last).  We also see that the local saddle connecting the two circles shown in $D_{1,0}$ cannot be a 1-1 bifurcation.  Thus the map $$\LCCi[1](D_1) \xrightarrow{\phi_R} \LCCi[2](D_2)$$ for the Reidemeister move, induced from \eqref{eq:phir}, behaves in precisely the same manner as it did for Lee complexes in $S^3$, and Rasmussen's checks in \cite{Rasmussen} apply virtually unchanged.
\end{proof}

\begin{lemma}[Reidemeister 2]\label{lem:Reid 2 on Lee gens}
	Let $D_1,D_2\subset \RP^2$ be two oriented link diagrams related by a single Reidemeister 2 move.  Then there exists a choice of dividing circles $\mc{C}_i$ for $D_{i,o}$ such that the chain homotopy equivalence of Theorem \ref{thm:Reidemeister moves give chain homotopy equivalences} descends to a map on homology
	\[\LHCi[1](D_1) \xrightarrow{\phi_R} \LHCi[2](D_2)\]
	which satisfies
	\[\phi_R([\LeegenCi[1]_o]) = \lambda [\LeegenCi[2]_o], \]
	where $\lambda$ is some invertible scalar in $\Q$ (and likewise for the opposite orientation $\oo$).
\end{lemma}

\begin{proof}
Here there are two cases to consider right away, based upon the relative orientations of the strands.  In the case that the strands are oriented in the same direction, we have the following local pictures for $D_1$ and $D_2$:
\[
D_1 = \ILtikzpic[xscale=.2,yscale=.2]{
	\draw[thick,->] (-60:3.5) to[out=120,in=-90] (-1.5,0) to[out=90,in=240] (60:3.5);
	\drawover[thick,->] {(-120:3.5) to[out=60,in=-90] (1.5,0) to[out=90,in=-60] (120:3.5);}
	\draw[dashed] (0,0) circle (3.5);
}
\quad , \quad
D_2 = \ILtikzpic[xscale=.2,yscale=.2]{
	\draw[thick,->] (-60:3.5) to[out=120,in=-90] (1,0) to[out=90,in=240] (60:3.5);
	\draw[thick,->] (-120:3.5) to[out=60,in=-90] (-1,0) to[out=90,in=-60] (120:3.5);
	\draw[dashed] (0,0) circle (3.5);
}\quad .
\]
In this case, we quickly see that we have matching oriented resolutions $D_{1,o}=D_{2,o}$ allowing us to once again choose the same dividing circles $\mc{C}_1=\mc{C}_2$.  Here we cannot rule out the possiblity of our local saddles being 1-1 bifurcations and so we cannot simply quote Rasmussen's check directly.  However, the same reasoning leads to the same check and the same result.  To wit, $\mc{C}_1=\mc{C}_2$ implies that $\LeegenCi[1]_o = \LeegenCi[2]_o$ in $C(D_{1,o})=C(D_{2,o})$, and the map $\LCCi[1](D_1) \xrightarrow{\phi_R} \LCCi[2](D_2)$ is built by collapsing the same acyclic complexes as in $S^3$, indicating that it is the identity map on the summand $C(D_{1,o})$ as desired.

In the case that the strands are oriented in opposite directions, we have the following local pictures for $D_1$ and $D_2$:
\[
D_1 = \ILtikzpic[xscale=.2,yscale=.2]{
	\draw[thick,<-] (-60:3.5) to[out=120,in=-90] (-1.5,0) to[out=90,in=240] (60:3.5);
	\drawover[thick,->] {(-120:3.5) to[out=60,in=-90] (1.5,0) to[out=90,in=-60] (120:3.5);}
	\draw[dashed] (0,0) circle (3.5);
}
\quad , \quad
D_2 = \ILtikzpic[xscale=.2,yscale=.2]{
	\draw[thick,<-] (-60:3.5) to[out=120,in=-90] (1,0) to[out=90,in=240] (60:3.5);
	\draw[thick,->] (-120:3.5) to[out=60,in=-90] (-1,0) to[out=90,in=-60] (120:3.5);
	\draw[dashed] (0,0) circle (3.5);
}\quad ,
\]
with oriented resolutions matching outside of the local pictures
\[
D_{1,o} = \ILtikzpic[xscale=.2,yscale=.2]{
	\draw[thick,<-] (-60:3.5) to[out=120,in=60] (-120:3.5);
	\draw[thick,->] (60:3.5) to[out=-120,in=-60] (120:3.5);
	\draw[thick] (0,0) circle (1.4);
	\draw[thick,<-] (-.1,1.4)--(.1,1.4);
	\draw[dashed] (0,0) circle (3.5);
}
\quad , \quad
D_{2,o} = \ILtikzpic[xscale=.2,yscale=.2]{
	\draw[thick,<-] (-60:3.5) to[out=120,in=-90] (1,0) to[out=90,in=240] (60:3.5);
	\draw[thick,->] (-120:3.5) to[out=60,in=-90] (-1,0) to[out=90,in=-60] (120:3.5);
	\draw[dashed] (0,0) circle (3.5);
}\quad ,
\]
where again we have drawn the inherited orientations despite the fact that we will choose the alternating orientations $o_{\mc{C}_i}$ for our set of auxiliary diagram choices.

We now let $\Gamma$ denote the connected trivalent graph consisting of $D_{2,o}$ together with the obvious arc along which surgery induces an oriented saddle leading to $D_{1,o}$ (after a further oriented birth cobordism).

\[\Gamma := \ILtikzpic[xscale=.2,yscale=.2]{
	\draw[thick,<-] (-60:3.5) to[out=120,in=-90] (1,0) to[out=90,in=240] (60:3.5);
	\draw[thick,->] (-120:3.5) to[out=60,in=-90] (-1,0) to[out=90,in=-60] (120:3.5);
	\draw[very thick,red] (-1,0)--(1,0);
	\draw[dashed] (0,0) circle (3.5);}
\]

If $\Gamma$ contains no essential circles, then it is local and $\mc{C}_1=\mc{C}_2$ can be chosen entirely disjoint from these local pictures, allowing us to use Rasmussen's check in \cite{Rasmussen} (note that once again, there are no 1-1 bifurcations available in the local picture $D_{1,o}$).  If $\Gamma$ does contain an essential circle, a case-by-case check (using the orientation of $D_{2,o}$) shows that in fact $D_{2,o}$ must contain an essential circle $\mc{C}_2$.  Surgery along the extra arc provides a method to see that $D_{1,o}$ must also contain a corresponding essential circle $\mc{C}_1$ so that the resulting alternating orientations $o_{\mc{C}_i}$ on $D_{i,o}$ lead to a check that is equivalent to the check in \cite{Rasmussen} (in this case there are no 1-1 bifurcations in the entire complex, as in Theorem \ref{thm:essential circle in resolutions gives sign assignment}).

\end{proof}

\begin{lemma}[Reidemeister 3]\label{lem:Reid 3 on Lee gens}
	Let $D_1,D_2\subset \RP^2$ be two oriented link diagrams related by a single Reidemeister 3 move.  Then there exists a choice of dividing circles $\mc{C}_i$ for $D_{i,o}$ such that the chain homotopy equivalence of Theorem \ref{thm:Reidemeister moves give chain homotopy equivalences} descends to a map on homology
	\[\LHCi[1](D_1) \xrightarrow{\phi_R} \LHCi[2](D_2)\]
	which satisfies
	\[\phi_R([\LeegenCi[1]_o]) = \lambda [\LeegenCi[2]_o], \]
	where $\lambda$ is some invertible scalar in $\Q$ (and likewise for the opposite orientation $\oo$).
\end{lemma}

\begin{proof}
For a Reidemeister 3 move, the local picture consists of 3 strands forming a half twist.  Up to the obvious rotational symmetry then, there are two cases to consider depending on the relative orientations of these three strands.  The first case is when the strands are oriented in the same way (up to rotation), and so we may assume without loss of generality that we have local pictures
\[
D_1=\ILtikzpic[xscale=.3,yscale=.3]{
				\draw[thick,<-] (135:3.5) -- (2,-1) to[out=-45,in=90] (-45:3.5);
				\drawover[thick,<-]{ (90:3.5) to[out=-90,in=90] (-2,0) to[out=-90,in=90] (-90:3.5); }
				\drawover[thick,<-]{ (45:3.5) to[out=-90,in=45] (2,1) -- (-135:3.5); }
				\draw[dashed] (0,0) circle (3.5);				
			}
\quad , \quad
D_2=\ILtikzpic[xscale=-.3,yscale=.3]{
				\draw[thick,<-] (45:3.5) to[out=-90,in=45] (2,1) -- (-135:3.5); 
				\drawover[thick,<-]{ (90:3.5) to[out=-90,in=90] (-2,0) to[out=-90,in=90] (-90:3.5); }
				\drawover[thick,<-]{ (135:3.5) -- (2,-1) to[out=-45,in=90] (-45:3.5); }
				
	\draw[dashed] (0,0) circle (3.5);				
}
\quad .
\]
In this case, much like for the first case of Reidemeister 2, we see that $D_{1,o}=D_{2,o}$, allowing us to choose equal dividing circles $\mc{C}_1=\mc{C}_2$ so that $\LeegenCi[1]_o=\LeegenCi[2]_o$ in $C(D_{1,o})=C(D_{2,o})$ and, regardless of the placement of the dividing curve, we have a check identical to the cases (a) and (b) in \cite[Proof of Proposition 2.3, Reidemeister 3]{Rasmussen} where the relevant map $\phi_R$ must be the identity on these summands.

Meanwhile, if the strands are oriented in alternating fashion, we have local pictures
\[
D_1=\ILtikzpic[xscale=.3,yscale=.3]{
	\draw[thick,<-] (135:3.5) -- (2,-1) to[out=-45,in=90] (-45:3.5);
	\drawover[thick,->]{ (90:3.5) to[out=-90,in=90] (-2,0) to[out=-90,in=90] (-90:3.5); }
	\drawover[thick,<-]{ (45:3.5) to[out=-90,in=45] (2,1) -- (-135:3.5); }
	\draw[dashed] (0,0) circle (3.5);				
}
\quad , \quad
D_2=\ILtikzpic[xscale=-.3,yscale=.3]{
	\draw[thick,<-] (45:3.5) to[out=-90,in=45] (2,1) -- (-135:3.5); 
	\drawover[thick,->]{ (90:3.5) to[out=-90,in=90] (-2,0) to[out=-90,in=90] (-90:3.5); }
	\drawover[thick,<-]{ (135:3.5) -- (2,-1) to[out=-45,in=90] (-45:3.5); }
	\draw[dashed] (0,0) circle (3.5);				
}
\]
leading to oriented resolutions identical outside of the local pictures
\[
D_{1,o} = \ILtikzpic[xscale=.3,yscale=.3]{
	\draw[thick,->] (90:3.5) to[out=-90,in=-45] (135:3.5);
	\draw[thick,->] (-135:3.5) to[out=45,in=90] (-90:3.5);
	\draw[thick,->] (-45:3.5) to[out=135,in=-135] (45:3.5);	
	\draw[thick] (-1,0) circle (1);
	\draw[thick,->] (-2,.1)--(-2,-.1);
	\draw[dashed] (0,0) circle (3.5);
}
\quad , \quad
D_{2,o} = \ILtikzpic[xscale=-.3,yscale=.3]{
	\draw[thick,->] (90:3.5) to[out=-90,in=-45] (135:3.5);
	\draw[thick,->] (-135:3.5) to[out=45,in=90] (-90:3.5);
	\draw[thick,->] (-45:3.5) to[out=135,in=-135] (45:3.5);	
	\draw[thick] (-1,0) circle (1);
	\draw[thick,->] (-2,.1)--(-2,-.1);
	\draw[dashed] (0,0) circle (3.5);
}
\quad .
\]

From here, we proceed in similar fashion to the proof for Reidemeister 2.  We let $\Gamma$ denote the connected trivalent graph consisting of $D_{2,o}$, with disjoint circle discarded, but with two extra arcs indicating how to reach $D_{1,o}$ via two oriented saddles (up to oriented death and birth).
\[
\Gamma := \ILtikzpic[xscale=-.3,yscale=.3]{
	\draw[thick,->] (90:3.5) to[out=-90,in=-45] node[pos=.3, inner sep=0] (N){} (135:3.5);
	\draw[thick,->] (-135:3.5) to[out=45,in=90] node[pos=.7, inner sep=0](S){} (-90:3.5);
	\draw[thick,->] (-45:3.5) to[out=135,in=-135] node[pos=.3, inner sep=0](SW){} node[pos=.7, inner sep=0](NW){} (45:3.5);	
	\draw[very thick, red] (N)--(NW);
	\draw[very thick, red] (S)--(SW);
	\draw[dashed] (0,0) circle (3.5);
}
\]
Once again, if $\Gamma$ contains no essential circles, it is local and $\mc{C}_1=\mc{C}_2$ can be chosen disjoint from these local pictures, and Rasmussen's check in \cite{Rasmussen} works directly (again, there are no 1-1 bifurcations avaliable in these local pictures).  If $\Gamma$ does contain an essential circle, a slightly more invovled case-by-case check (again using the orientation of $D_{2,o}$) shows that $D_{2,o}$ itself must have contained an essential circle $\mc{C}_2$.  The orientable cobordism to reach $D_{1,o}$ again indicates how to define and orient $\mc{C}_1$ (contained in $D_{1,o}$) leading to a check equivalent to the one in \cite{Rasmussen}.  (Note that the disjoint circles in the $D_{i,o}$ will be on `opposite sides' of the $\mc{C}_i$ in these cases, but they will also be oriented oppositely from each other---this ensures that the computations are essentially the same as they were in $S^3$.)
\end{proof}

\begin{lemma}[Reidemeister 4 and 5]\label{lem:Reid 4 and 5 on Lee gens}
	Let $D_1,D_2\subset \RP^2$ be two oriented link diagrams related by a single Reidemeister 4 or Reidemeister 5 move.  Then there exists a choice of dividing circles $\mc{C}_i$ for $D_{i,o}$ such that the chain homotopy equivalence of Theorem \ref{thm:Reidemeister moves give chain homotopy equivalences} descends to a map on homology
	\[\LHCi[1](D_1) \xrightarrow{\phi_R} \LHCi[2](D_2)\]
	which satisfies
	\[\phi_R([\LeegenCi[1]_o]) = [\LeegenCi[2]_o],\]
	and likewise for the opposite orientation $\oo$.
\end{lemma}
\begin{proof}
As mentioned in the proof of Theorem \ref{thm:Reidemeister moves give chain homotopy equivalences}, the fourth and fifth Reidemeister moves induce chain isomorphisms via isotopies (in $\RP^2$) of resolution diagrams, including the oriented resolution diagrams $D_{i,o}$.  Then for any choice of $\mc{C}_1$, we use this isotopy to define $\mc{C}_2$ so that the map clearly sends $\LeegenCi[1]_o$ directly to $\LeegenCi[2]_o$ (even before passing to homology).
\end{proof}

\subsection{The $s$-invariant of a link and basic properties}
\label{sec:mirror and conn sum}
The results of the previous section give us the naturality required to define the $s$-invariant of a link rather than a link diagram.

\begin{definition}\label{def:s invt of link}
The \emph{$s$-invariant} of an oriented link $L\subset\RP^3$, denoted $s(L)$, is defined to be the $s$-invariant of any oriented link diagram $D$ for $L$ (as in Definition \ref{def:s invariant for a diagram}).  
\end{definition}

The $s$-invariant is well-defined by Proposition \ref{prop:Lee gens under change of div circ} together with Lemmas \ref{lem:Reid 1 on Lee gens}-\ref{lem:Reid 4 and 5 on Lee gens}.

As a corollary of Theorem \ref{thm:local link is same as S3}, we have the following basic properties.

\begin{proposition}\label{prop:local links}
Let $L_\ell\subset B^3\subset \RP^3$ be a local link in $\RP^3$.  Then 
\[s_{\RP^3}(L_\ell)=s_{S^3}(L_\ell),\]
where $s_{\RP^3}$ and $s_{S^3}$ denote the $s$-invariants for links in $\RP^3$ and $S^3$ respectively.
\end{proposition}
\begin{proof}
	This is a direct consequence of Theorem \ref{thm:local link is same as S3}, together with the fact that for local links a local diagram $D$ can be chosen with a dividing curve disjoint from the ball containing $D$.
\end{proof}

Furthermore, if $L_\ell\subset B^3\subset \RP^3$ is a local link, and $L\subset \RP^3$ is a second link (maybe not local), then we may define the disjoint union
\[L\sqcup L_\ell \subset \RP^3\# B^3 \subset \RP^3 \# S^3 \cong \RP^3\]
and the connected sum
\[L\# L_\ell \subset \RP^3\# B^3 \subset \RP^3 \# S^3 \cong \RP^3.\]

\begin{proposition}\label{prop:disj unions and conn sums}
Let $L_\ell$ and $L$ be links as above.

$(a)$ For the disjoint union, we have
\[   s_{\RP^3}(L\sqcup L_\ell) = s_{\RP^3}(L) + s_{\RP^3}(L_\ell) - 1 = s_{\RP^3}(L) + s_{S^3}(L_\ell) -1.\]

$(b)$ For the connected sum, we have
\[  s_{\RP^3}(L\#L_\ell) = s_{\RP^3}(L) + s_{\RP^3}(L_\ell) = s_{\RP^3}(L) + s_{S^3}(L_\ell).\]
\end{proposition}
\begin{proof}
Both of these formulas can be proven in the same manner as for links in $S^3$ (see \cite[Proposition 3.12]{Rasmussen} for the connect sum of knots, and \cite[Lemma 7.1]{BW}, \cite[Theorem 7.1]{MMSW} for disjoint unions and connect sums of links; all of the connect sum arguments utilize the exact sequence for a crossing which extends to our setting as in Proposition \ref{prop:skein exact seq}).  The only added wrinkle involves a consistent choice of dividing circles $\mc{C}_\#,\mc{C}_{\sqcup}$ for $L\# L_\ell$ and $L\sqcup L_\ell$ which appear in the proof. Note that after choosing a local diagram $D_\ell\subset B^2$ for $L_\ell$ and a diagram $D \subset \RP^2$ for $L$, we obtain a diagram $D \sqcup D_{\ell}$ for $L \sqcup L_\ell$ by inserting $D_{\ell} \subset B^2$ in $\RP^2$ away from the strands in $D$. From here we also obtain a diagram $D \# D_{\ell}$ for $L \# L_\ell$ by connecting the two diagrams via two parallel paths. The choice of dividing circles is unique for all the relevant diagrams in the case when $L$ (and hence also $L\sqcup L_\ell$ and $L\#L_\ell$) are of class-1. When they are of class-0, given a dividing circle $\mc{C}$ for a resolution $D_u$ of $D$, we can isotope $\mc{C}$ to be disjoint from the disk $B^2$ that contains $D_\ell$, and also disjoint from the connecting 
parallel paths. Then, we can use $\mc{C}=\mc{C}_{\sqcup}=\mc{C}_\#$ as a dividing circle for the corresponding resolutions of $D\sqcup D_{\ell}$ and $D \# D_\ell$ as well. With these compatible choices of dividing circles, the proofs go through as in the planar case.
\end{proof}

\begin{proof}[Proof of Theorem~\ref{thm:local}] 
Part (a) is Proposition~\ref{prop:local links}, and part (b) is the second statement in Proposition~\ref{prop:disj unions and conn sums}.
\end{proof}

\begin{corollary}\label{cor:smin and smax}
The two filtration levels $\q\left(\left[ \LeegenC_o + \LeegenC_{\oo} \right]\right)$ and $\q\left(\left[ \LeegenC_o - \LeegenC_{\oo} \right]\right)$ used to define $s(D)$ (for a link diagram $D$) differ by two.
\end{corollary}
\begin{proof}
The proof is analogous to the proof given for knot diagrams in \cite[Proposition 3.3]{Rasmussen} (extended to links in \cite[Section 6.1]{BW}). It uses the connected sum of a given knot diagram with the standard crossingless (and in our case, local) diagram for the unknot.
\end{proof}

We also have the expected behavior for mirrors. 

\begin{proposition}\label{prop:mirrors}
Let $K\subset\RP^3$ be any knot in $\RP^3$. Denote by $m(K)$ the mirror, by $K^r$ the reverse, and let $-K=m(K^r)$.  Then 
$$s(-K)=s(m(K))=-s(K^r) = -s(K).$$
\end{proposition}
\begin{proof}
For the statement about mirrors, the proof is similar to that in \cite{Rasmussen}, noting that taking the mirror of a knot diagram gives a well-defined operation on dividing curves as well. For the statement about the reverse, note that changing the orientation of $K$ switches the orientations $o$ and $\oo$, but the definition of the $s$-invariant uses both $o$ and $\oo$ on an equal footing.
\end{proof}

\begin{remark}
Similarly to what happens in $S^3$, for links $L \subset \RP^3$ with at least two components, we still have $s(L^r)=s(L)$, but the quantity $s(-L)=s(m(L))$ is not determined by $s(L)$. Compare \cite[Section 7.2]{MMSW}.
\end{remark}

\subsection{The family of invariants $s_{\tau}$ for $\tau\in[0,2]$ are all the same}

For annular links $L\subset S^1\times D^2$, the Lee differential is filtered with respect to the grading $j-\tau k$ for all $\tau\in[0,2]$.  In this way one can define an entire family of Rasmussen invariants $s_{\tau}(L)$ depending on the choice of filtration \cite{GLW}.  The presence of an involution $\Theta$ in that setting shows that $s_{1-\tau}(L)=s_{1+\tau}(L)$ for such links, but one can still have different $s_{\tau}$-invariants for different values of $\tau\in[0,1]$.

In much the same way, our Lee differential for links $L\subset \RP^3$ is also filtered with respect to the grading $j-\tau k$ for all $\tau\in [0,2]$, giving rise to an entire family of Rasmussen invariants  $s_{\tau}(L)$.  Once again we have an involution $\Theta$ (see Proposition \ref{prop:symmetry for class 1}) which ensures that $s_{1-\tau}(L)=s_{1+\tau}(L)$ for all $L\subset\RP^3$ (the proof of this fact is entirely analogous to the annular case).  However, in this section we will show that the invariants $s_{\tau}(L)$ for $\tau\in[0,2]$ are in fact all equal.  Conceptually, this reflects the fact that, for links in $\RP^3$, the $k$-grading is confined to the values $k\in\{-1,0,1\}$, forcing the Lee complex to be contained in a very thin strip in the $(j,k)$-plane. Together with the presence of the involution $\Theta$, this prevents any sort of interesting variance in $(j-\tau k)$-filtration levels as $\tau$ varies from $0$ to $2$. (This should not be too surprising in view of Remark~\ref{rem:bigraded}.)

To begin, we fix an oriented link diagram $(D,o)$ for a link $L\subset\RP^3$ and note that, if $L$ is a class-0 link, then the $k$-grading in $LC(L)$ is identically zero, and the statement is trivial.  Thus we may assume that $D$ is the diagram for a class-1 link, with $k$ gradings of generators either $1$ or $-1$, and Proposition \ref{prop:symmetry for class 1} providing the involution $\Theta$ which interchanges the labels $\oone$ and $\ox$ on the essential circle in each resolution $D_v$ of $D$.  We continue to let $\LeegenC_o,\LeegenC_{\oo} \in C(D_o)$ denote the Lee generators for $(D,o)$ with respect to a fixed choice of dividing circle $\mc{C}$ (which in this case is a choice of orientation for the unique essential circle in $D_o$).

\begin{lemma}[{{cf. \cite[Lemma 3.5]{Rasmussen}}}]
\label{lem:split LC into even and odd}
The complex $LC(D)$ splits as a direct sum
\[LC(D) \cong LC(D)_e\oplus LC(D)_o\]
where $LC(D)_e$ (respectively $LC(D)_o$) is spanned by the homogeneous basis elements in quantum grading $j \equiv |L| \!\pmod 4$ (respectively $j \equiv |L| +2 \!\pmod 4$), with $|L|$ denoting the number of components of $L$.

Moreover, one of the two elements $\LeegenC_o\pm\LeegenC_{\oo}$ is in $LC(D)_e$, and the other is in $LC(D)_o$. 
\end{lemma}
\begin{proof}
The proof given in \cite[Lemma 3.5]{Rasmussen} adapts readily to our setting.
\end{proof}

\begin{lemma}\label{lem:Theta swaps so+ and so-}
The involution $\Theta$ on $LC(D)$ interchanges the two summands $LC(D)_e$ and $LC(D)_o$, and moreover satisfies
\[\Theta(\LeegenC_o+\LeegenC_{\oo}) = \pm(\LeegenC_o-\LeegenC_{\oo}).\]
\end{lemma}
\begin{proof}
The statement on interchanging the summands is immediate from the definition since $\oone$ and $\ox$ differ in $j$-degree by two.  In the Lee basis we see $\Theta(\oa)=\oa$ and $\Theta(\ob)=-\ob$, from which the second statement follows since there is exactly one essential circle with a label of either $\oa$ or $\ob$ (the sign is determined by which of $\LeegenC_o,\LeegenC_{\oo}$ has label $\oa$, while the other has $\ob$).
\end{proof}

Let us use the notation $\q_{\tau}([z])$ to denote the $(j-\tau k)$-filtration level of a homology class $[z]\in\LH(D)$ for $\tau\in[0,2]$.  In light of Lemmas \ref{lem:split LC into even and odd} and \ref{lem:Theta swaps so+ and so-}, we may define $w\in LC(D)_e$ to be whichever of $\LeegenC_o\pm \LeegenC_{\oo}$ is contained in $LC(D)_e$, so that our desired filtrations levels $s_\tau(D)$ satisfy
\[s_{\tau}(D)= \frac{\q_{\tau}([w]) + \q_\tau([\Theta(w)])}{2}.\]
Moreover, since the splitting of Lemma \ref{lem:split LC into even and odd} respects gradings, we are free to compute filtration levels within each summand:
\[\q_\tau([w]) = \q_\tau([w]_e), \quad\quad \q_\tau([\Theta(w)]) = \q_\tau([\Theta(w)]_o),\]
where $[w]_e := [w] \cap LC(D)_e$ and $[\Theta(w)]_o := [\Theta(w)]\cap LC(D)_o$.

To track these filtration levels, we will consider the sets
\[\Lambda_e := \{j \,|\, j\equiv |L| \!\!\!\!  \pmod 4\}\times \{-1,1\} \quad \text{and} \quad \Lambda_o:= \{j \,|\, j\equiv |L|+2 \!\!\!\!  \pmod 4\} \times \{-1,1\}\]
of possible bidegrees of homogeneous generators in $LC(D)_e$ and $LC(D)_o$ respectively.  Each of $\Lambda_e,\Lambda_o$ can be endowed with a total order inherited from the lexicographic order on $\Z\times\{-1,1\}$, where we declare the ordering $-1 > 1$ on the second factor. This allows for filtration levels to be tracked via maps
\[\Lambda_\bullet \xrightarrow{\q^\bullet_\tau} \R\]
defined as $\q_\tau^\bullet(j,k) := j-\tau k$ for $\bullet\in\{e,o\}$ and $\tau\in[0,2]$.

In this way we can compute $\q_\tau([w]_e)$ as
\begin{equation}\label{eq:def of filt level} \q_\tau([w]_e) = \max_{w'\sim w} \min \q^e_\tau(\supp_e(w')),
\end{equation}
where $w'$ is homologous to $w$ in $LC(D)_e$, and  $\supp_e(w')\subset \Lambda_e$ denotes the set of bidegrees of all homogeneous summands appearing in the unique representation of $w'$ in the homogeneous basis for $LC(D)_e$.  Similarly we can compute
\begin{equation}\label{eq:def of filt level of involution}
\begin{aligned}
\q_\tau([\Theta(w)]_o) &= \max_{z' \sim \Theta(w)} \min \q_\tau^o( \supp_o(z') )\\
&= \max_{w'\sim w} \min \q_\tau^o ( \supp_o( \Theta(w')) ),
\end{aligned}
\end{equation}
where $\supp_o( z') \subset LC(D)_o$ is defined analogously to $\supp_e$.  To obtain the second equality we have used the fact that $\Theta$ is a chain involution to set $w':=\Theta(z')$ and take the max over all $w' \sim w$ in $LC(D)_e$.

The key point about tracking filtration levels in this way rests in the following two lemmas.

\begin{lemma}\label{lem:filt levels are order preserving}
Using the notation established above, the maps $\Lambda_\bullet\xrightarrow{\q_\tau^\bullet}\R$ for $\bullet\in\{e,o\}$ and $\tau\in[0,2]$ are order-preserving (although not strictly in the case that $\tau\in\{0,2\}$).
\end{lemma}
\begin{proof}
A simple arithmetic check.
\end{proof}

\begin{lemma}\label{lem:involution is order preserving}
The maps
\[\Lambda_e \xrightarrow{\Theta^e_o} \Lambda_o, \quad \quad \Lambda_o \xrightarrow{\Theta^o_e} \Lambda_e\]
defined by
\[\Theta_*^\bullet (j,k) := (j-2k,-k)\]
for $(\bullet,*)\in\{(e,o),(o,e)\}$ are order-preserving mutual inverses satisfying, for any $z\in LC(D)_\bullet$,
\[\supp_* (\Theta(z)) = \Theta_*^\bullet ( \supp_\bullet(z)).\]
\end{lemma}
\begin{proof}
Another arithmetic check, using the fact that $\Theta$ is a chain involution which affects the bigradings as indicated for the last claim.
\end{proof}

We are now ready to state and prove the main theorem of this section. (This proof was suggested by the referee.)

\begin{theorem}\label{thm:no extra s-invariants}
Let $D\subset\RP^2$ denote a fixed diagram for an oriented link in $\RP^3$.  For $\tau\in[0,2]$, let $s_{\tau}(D)$ denote the Rasmussen invariant of $D$ using the $(j-\tau k)$-filtration for the Lee complex $LC(D)$.  Then $s_{\tau}(D)=s_0(D)=s(D)$ is independent of $\tau$, and thus there is only one such Rasmussen invariant for links in $\RP^3$.
\end{theorem}
\begin{proof}
As mentioned above, if $D$ is the diagram for a class-0 link, then the $k$-grading in $LC(D)$ is identically zero and the statement is trivial.  Otherwise, we assume $D$ is the diagram for a class-1 link. We continue to use the splitting $LC(D)\cong LC(D)_e\oplus LC(D)_o$ and other notations introduced throughout this section.  In particular, letting $w \in \{\LeegenC_o \pm \LeegenC_{\oo}\}$ denote whichever generator is in $LC(D)_e$, we have
\[2s_{\tau}(D)= \q_{\tau}([w]_e) + \q_\tau([\Theta(w)]_o).\]
Combining Equation \eqref{eq:def of filt level} with Lemma \ref{lem:filt levels are order preserving} allows us to conclude that
\[ \q_\tau([w]_e) = \q^e_\tau (\max_{w'\sim w} \min \supp_e(w')),\]
where the minimum and maximum are taken with respect to the total ordering on $\Lambda_e$.  Similarly we can combine Equation \ref{eq:def of filt level of involution} with Lemmas \ref{lem:filt levels are order preserving} and \ref{lem:involution is order preserving} to see that
\[ \q_\tau([\Theta(w)]) = \q_\tau^o(\Theta_o^e( \max_{w'\sim w} \min \supp_e(w'))).\]
If we set $(j_0,k_0):=(\max_{w'\sim w}\min\supp_e(w')) \in \Lambda_e$, then we have
\begin{align*}
2s_\tau(D) &= \q_{\tau}([w]_e) + \q_\tau([\Theta(w)]_o)\\
&= \q_\tau^e(j_0,k_0) + \q_\tau^o (\Theta^e_o(j_0,k_0))\\
&= j_0-\tau k_0 + (j_0-2k_0 +\tau k_0)\\
&= 2(j_0-k_0),
\end{align*}
which is independent of $\tau$.
\end{proof}

\section{Cobordisms}
\label{sec:cobordisms}

The goal of this section is to prove Theorem \ref{thm:BW} from the Introduction.  That is to say, we wish to bound the genus of oriented link cobordisms $\Sigma\subset\RP^3\times I$ from $L_1$ to $L_2$ using the $s$-invariant of the two boundary links.  Following \cite{Rasmussen,BW}, we will do this by first assigning a map $\phi_\Sigma:LC^*(L_1)\rightarrow LC^*(L_2)$ to $\Sigma$, and then analyzing its effect on Lee generators and its filtration degree.

As in \cite{Rasmussen}, any cobordism can be decomposed into a sequence of single Reidemeister and Morse moves, and so we focus on these individually first.  In the case that $\phi_\Sigma$ corresponds to a single Reidemeister move, this analysis has already been carried out in Section \ref{sec:Reid moves}.  In Section \ref{sec:Morse moves} we present the similar (and simpler) analysis for maps assigned to single Morse moves, before proceeding to the general case and proof of Theorem \ref{thm:BW} in Section \ref{sec:general cobs on Lee gens}.  Throughout we will continue to use the notation  $\LCC(D)$ for the Lee complex of a diagram $D$ using a set of auxiliary diagram choices which uses the orientation $o$ of $D$ to locally orient all crossings, and which orients $D_o$ using $o_{\mc{C}}$ via a choice of dividing circle $\mc{C}$.  We then have Lee generators denoted $\LeegenC_o,\LeegenC_{\oo}\in \LCC(D)$ generating two summands of the Lee homology $LH^*_{\mc{C}}(D)\cong \Q^{|O(L)|}$ as in Theorem \ref{thm:Lee homology of diagram}.

\subsection{The effect of elementary Morse moves on Lee generators}
\label{sec:Morse moves}

The maps $\phi_\Sigma$ for single Morse moves are defined in precisely the same way as for Khovanov and Lee complexes in $S^3$.

\begin{definition}\label{def:elem Morse map}
	Let $\Sigma:L_1\rightarrow L_2$ be a link cobordism in $\RP^3\times I$ consisting of a single elementary Morse move (birth, death, or saddle).  Fix two link diagrams $D_1,D_2$ for $L_1,L_2$.  Then the induced chain map on the deformed complexes
	\[ \KC_d^*(D_1) \xrightarrow{\phi_\Sigma}  \KC_d^*(D_2)\]
	is defined on each resolution of $D_1$ as follows.  That is to say, a birth places a label $1$ on the new (local) circle, as in
	\[ z \mapsto z\otimes 1\]
	for any generator $z\in  \KC_d^*(D_1)$.	A death acts as the indicator function for $X$ on the dying (local) circle, as in
	\[ z\otimes 1 \mapsto 0,\quad z\otimes X \mapsto z.\]
	Finally, a saddle acts as $\pm m, \pm \Delta,$ or $0$ depending on its bifurcation type, just as it would for an edge map in a cube of resolutions.  Specializing to the Lee complex, we see that each of these maps is filtered of degree equal to the Euler characteristic of $\Sigma$, just as for $S^3$ (with the exception of the unorientable saddle, which gives the zero map).
\end{definition}

Because birth and death cobordisms are entirely local, only applying to local circles in $\RP^2$, Lemmas \ref{lem:births on Lee gens} and \ref{lem:deaths on Lee gens} below are immediate from the definition in the same way as they are for links in $S^3$.

\begin{lemma}[Birth maps]\label{lem:births on Lee gens}
	Let $D_1,D_2\subset \RP^2$ be two link diagrams related by a single birth cobordism $\Sigma$.  Fix an orientation $o_1$ on $D_1$, which induces two orientations $o_2',o_2''$ on $D_2$ according to the choice of orientation for the newly birthed circle.  Then there exists a choice of dividing circles $\mc{C}_i$ far from the point of birth for $D_{i,o_i}$ such that the induced map on Lee homology
	\[\LHCi[1](D_1) \xrightarrow{\phi_\Sigma} \LHCi[2](D_2)\]
	satisfies 
	\[\phi_\Sigma([\LeegenCi[1]_{o_1}]) = \frac{1}{2} ( [\LeegenCi[2]_{o_2'}] + [\LeegenCi[2]_{o_2''}] ),\]
	and likewise for the opposite orientation $\oo_1$ on $D_1$.
\end{lemma}

\begin{lemma}[Death maps]\label{lem:deaths on Lee gens}
	Let $(D_1,o_1),(D_2,o_2)\subset \RP^2$ be two oriented link diagrams related by a single death cobordism $\Sigma$.  Then there exists a choice of dividing circles $\mc{C}_i$ for $D_{i,o_i}$ such that the induced map on Lee homology
	\[\LHCi[1](D_1) \xrightarrow{\phi_\Sigma} \LHCi[2](D_2)\]
	satisfies
	\[\phi_\Sigma([\LeegenCi[1]_{o_1}]) = [\LeegenCi[2]_{o_2}], \]
	and likewise for the opposite orientations $\oo_1,\oo_2$.
\end{lemma}

Meanwhile, saddle cobordisms require slightly more work, most of which has already been handled while considering Reidemeister 2 moves.

\begin{lemma}[Saddle maps]\label{lem:saddles on Lee gens}
	Let $D_1,D_2\subset \RP^2$ be two link diagrams related by a single saddle cobordism $\Sigma$.  If $\Sigma$ is not orientable, then the induced map on Lee homology
	\[\LH^*(D_1) \xrightarrow{\phi_\Sigma} \LH^*(D_2)\]
	is the zero map.  Otherwise we may fix an orientation of $\Sigma$ inducing orientations $o_i$ on the $D_i$.  Then there exists a choice of dividing circles $\mc{C}_i$ for $D_{i,o_i}$ such that
	\[\phi_\Sigma([\LeegenCi[1]_{o_1}]) = \lambda [\LeegenCi[2]_{o_2}],\]
	where $\lambda$ is some invertible scalar in $\Q$ (and likewise for the opposite orientations $\oo_1,\oo_2$).
\end{lemma}
\begin{proof}
As in \cite{Rasmussen}, we let $o_1'$ denote an arbitrary orientation on $D_1$ and see how it compares to possible orientations (or the lack thereof) for $\Sigma$.  Just as in the proof of Lemma \ref{lem:Reid 2 on Lee gens} for the Reidemeister 2 move, we let $\Gamma$ denote the connected trivalent graph incorporating the `support' of the saddle move
\[\Gamma := \ILtikzpic[xscale=.2,yscale=.2]{
	\draw[thick] (-60:3.5) to[out=120,in=-90] (1,0) to[out=90,in=240] (60:3.5);
	\draw[thick] (-120:3.5) to[out=60,in=-90] (-1,0) to[out=90,in=-60] (120:3.5);
	\draw[very thick,red] (-1,0)--(1,0);
	\draw[dashed] (0,0) circle (3.5);}.
\]
Then we repeat the analysis of $\Gamma$ used in the proof of Lemma \ref{lem:Reid 2 on Lee gens}, which shows that, in almost all cases, a dividing circle $\mc{C}_1$ for $D_{1,o_1'}$ can be chosen (either disjoint from $\Gamma$ or entirely contained in $D_{1,o_1'}$) such that Rasmussen's analysis in \cite{Rasmussen} passes through virtually unchanged.  There is one potential case here that did not appear in the Reidemeister 2 proof, where $\Gamma$ contains an essential circle which passes along the saddle arc from one corner of the local picture above to the opposite corner (this case was disallowed in the Reidemeister 2 proof due to the known orientations on the diagram $D_{1,o}$ being considered there).  However, it is easy to see that this case corresponds to the saddle inducing a 1-1 bifurcation on $D_{1,o_1'}$, and thus a zero map on this resolution.  This also implies that $\Sigma$ was incompatible with this orientation (and in fact was not orientable at all), verifying the claim in this case as well.
\end{proof}

\subsection{General cobordisms}\label{sec:general cobs on Lee gens}
In this section we use the lemmas of Sections \ref{sec:Reid moves} and \ref{sec:Morse moves} to conclude (following \cite{Rasmussen,BW,MMSW}) that oriented cobordisms between links $L_1$ and $L_2$ lead to bounds on the difference $s(L_1)-s(L_2)$. Compare the following theorem with \cite[Theorem 3.8]{MMSW}.

\begin{theorem}\label{thm:general cobs on Lee gens}
	Let $\Sigma\subset I\times \RP^3$ denote a cobordism between links $L_1\subset \{0\}\times \RP^3$ and $L_2\subset \{1\}\times\RP^3$ with link diagrams $D_1,D_2\subset\RP^2$.  Fix an orientation $o_1$ on $L_1$, and let $O(\Sigma,o_1)$ denote the set of orientations $o$ of $\Sigma$ whose induced orientation $o|_1$ on $L_1$ is $o_1$.  For any such orientation $o\in O(\Sigma,o_1)$, let $o|_2$ denote the induced orientation on $L_2$.
	
	Then there exists a choice of dividing circle $\mc{C}_1$ for $D_{1,o_1}$, and a choice of single dividing circle $\mc{C}_2$ for all of the various $D_{2,o|_2}$, together with an induced map on Lee homology $\LHCi[1](D_1) \xrightarrow{\phi_\Sigma} \LHCi[2](D_2)$ which is a filtered map of filtration degree $\chi(\Sigma)$ (the Euler characteristic of $\Sigma$) satisfying
	\begin{equation}\label{eq:gen cob on Lee gens}
		\phi_\Sigma([\LeegenCi[1]_{o_1}]) = \sum_{o\in O(\Sigma,o_1)} \lambda_o [\LeegenCi[2]_{o|_2}],
	\end{equation}
	where for each $o\in O(\Sigma,o_1)$, $\lambda_o$ is a unit in $\Q$.	
\end{theorem}
\begin{proof}
As in \cite{Rasmussen}, we begin by decomposing $\Sigma$ into a sequence of elementary cobordisms
\[L_1=M_1 \xrightarrow{\Sigma_1} M_2 \xrightarrow{\Sigma_2} \cdots \xrightarrow{\Sigma_{m-1}} M_m = L_2,\]
where each intermediate link $M_i$ has a diagram $E_i\subset \RP^2$.  Then for each elementary $\Sigma_i$ and orientation $o_i$ on $M_i$, one of the Lemmas \ref{lem:Reid 1 on Lee gens}-\ref{lem:Reid 4 and 5 on Lee gens} (for Reidemeister moves, which have filtration degree zero by Theorem \ref{thm:Reidemeister moves give chain homotopy equivalences}) or \ref{lem:births on Lee gens}-\ref{lem:saddles on Lee gens} (for Morse moves) provides a choice of dividing circles which we denote $\mc{C}_i$ for $E_i$ and $\mc{C}_{i+1}'$ for $E_{i+1}$ giving $\LHCi[i](E_i)\xrightarrow{\phi_{\Sigma_i}} \LH^*_{\mc{C}_{i+1}'}(E_{i+1})$ satisfying
\[\phi_{\Sigma_i} \left(\left[\LeegenCi[i]_{o_i}\right]\right) = \sum_{o\in O(\Sigma,o_i)} \lambda_o \left[\Leegen_{o|_{i+1}}^{\mc{C}_{i+1}'}\right].\]
(Note that in all cases but for births, this sum for an elementary cobordism consists of either one term or none at all.)

We then interweave these maps $\phi_{\Sigma_i}$ with maps $\phi$ coming from changing dividing circles from $\mc{C}_i'$ to $\mc{C}_i$ via Theorem \ref{thm:  KC is well-defined}.  These maps have filtration degree zero and have the proper behavior on Lee generators via Proposition \ref{prop:Lee gens under change of div circ}, allowing us to build our desired map $\phi_\Sigma$ as the composition
\[
\LHCi[1](E_1) \xrightarrow{\phi_{\Sigma_1}} \LH^*_{\mc{C}_2'} (E_2)
\xrightarrow{\phi} \LHCi[2](E_2) \xrightarrow{\phi_{\Sigma_2}} \LH^*_{\mc{C}_3'}(E_3)
\xrightarrow{\phi} \LHCi[3](E_3) \rightarrow \cdots
\]
until we finally reach $\LH^*_{\mc{C}_m'}(E_m)$, with $E_m=D_2$.  This allows us to make the choice of $\mc{C}_2:=\mc{C}_m'$, after which the verification of the claim reduces to the same inductive double-summation argument described in \cite[proof of Theorem 3.8]{MMSW}. \end{proof}

\begin{proof}[Proof of Theorem \ref{thm:BW}]
The proof is identical to the proof of \cite[Theorem 1.5]{MMSW}, itself a slight reformulation of the original argument in \cite{Rasmussen}.  In summary, the connectedness hypothesis ensures that the sum in Equation \eqref{eq:gen cob on Lee gens} contains precisely one term, so that in such cases the map $\phi_\Sigma$ is an isomorphism on Lee homology maintaining Lee generators, whose filtration levels define $s(L_i)$.
\end{proof}

\begin{proof}[Proof of Corollary~\ref{cor:s}]
Apply Theorem \ref{thm:BW} to cobordisms $\Sigma \subset I \times \RP^3$ between the knot $K$ and the unknot $U_\alpha$ of the same class $\alpha \in \{0, 1\}$ as $K$. Observing that $s(U_0)=s(U_1)=0$, we get $s(K) \leq 2g_4(K)$. Reversing the cobordism gives the inequality $-s(K) \leq 2g_4(K)$.
\end{proof}

\section{Properties and applications}
\label{sec:app}

\subsection{Positive links}
An oriented link $L \subset \RP^3$ is called {\em positive} if it admits a diagram $D \subset \RP^2$ with only positive crossings. Given such a diagram, we let $n$ be the number of its crossings, and let $r$ be the number of circles in its oriented resolution $D_o$.

\begin{proof}[Proof of Theorem~\ref{thm:positive}]
This is similar to Rasmussen's computation of the $s$-invariant for positive knots in $S^3$, done in \cite[Section 5.2]{Rasmussen}. In a positive diagram $D$, the oriented resolution $D_o$ is arrived at by choosing the $0$-resolution at each crossing. After choosing a dividing circle $\mc{C}$ for $D_o$, we see that the Lee generator $\LeegenC_o$ lives in the minimal quantum degree $n-r$ among all generators of the Khovanov complex. Therefore, $\LeegenC_o$ is not  homologous to any other class, so $s(K)-1= \s_{\min}(D)=\q \left(\left[ \LeegenC_o \right]\right) =  n-r.$
\end{proof}

\begin{corollary}
\label{cor:g4positive} 
The slice genus of a positive knot $K \subset \RP^3$ equals $(n-r+1)/2$.
\end{corollary}

\begin{proof}
Seifert's algorithm for finding Seifert surfaces for knots in $S^3$ can be applied to diagrams $D \subset \RP^2$ of knots in $\RP^3$. Let $D_{o}$ be the oriented resolution of $D$. With the exception of the essential circle (which exists when $K$ is class-1), the circles in $D_{o}$ bound oriented disks, which in $D$ are connected by oriented saddles. They can also be connected by oriented saddles to the essential circle.

When $K$ is class-0, we obtain an orientable surface in $\RP^3$ with boundary $K$ whose genus is $g=(n-r+1)/2$. By pushing this surface into $I \times \RP^3$ so that $K$ lives in $ \{0\} \times \RP^3$, and then connecting the surface via a tube to $\{1\} \times \RP^3$, we get a cobordism of genus $g$ from $K$ to the unknot $U_0$. When $K$ is class-1, we get instead a genus $g$ cobordism from $K$ to the essential circle in $\RP^2 \subset \RP^3$, which is the class-1 unknot $U_1$. In either case, we obtain the inequality $2g_4(K) \leq n-r+1$. Theorem~\ref{thm:positive} and Corollary~\ref{cor:s} give the opposite inequality, and the desired result follows.
\end{proof}

\begin{remark}
Corollary~\ref{cor:g4positive} can also be proved using Rasmussen's original results for the $s$-invariant in $S^3$, by applying them to the lift $\tK$ of $K$ in the double cover considered in the next section. 
\end{remark}

\begin{remark}
\label{rem:positive links}
The proof of Theorem~\ref{thm:positive} extends to positive links $L \subset \RP^3$, yielding the same formula $s(L)=n-r+1$. (Compare Proposition 5.2 and Remark 5.3 in \cite{Kawamura}.) 
\end{remark}

\subsection{Relation to freely 2-periodic knots}
A link $\tL \subset S^3$ is called {\em freely 2-periodic} if it is invariant under the involution $x \mapsto -x$ on $S^3 \subset \R^4$. The quotient of $\tL$  is a link $L  \subset \RP^3 = S^3/\sim$. Conversely, given a link $L \subset \RP^3$, we can consider its lift $\tL$ in $S^3$, which is freely 2-periodic. Note that a connected component $K \subset L$ lifts to a single component $\tK \subset \tL$ if $K$ is of class-1, and it lifts to a link $\tK \subset \tL$ of two components if $K$ is of class-0.

Starting with a diagram $D \subset \RP^2$ representing $L$, we can obtain diagrams of $\tL$ as follows. 
Let us view $D$ as the ``projective closure'' of an $(m, m)$-tangle $T \subset B^2$; that is, when we self-glue the boundary of $B^2$ in an antipodal fashion to obtain $\RP^2$, we identify the $m$ left endpoints of $T$ with its $m$ right endpoints, reversing their order. (Note that the parity of $m$ gives the class of $L$.)
Then, as explained in \cite[Theorem 2.3]{Manfredi}, we get a planar diagram for $\tL$ by taking the planar  closure of the tangle $T \circ \Delta \circ T \circ \Delta^{-1}$, where $\Delta$ is the half-twist on $m$ strands; see Figure~\ref{fig:tDprelim}.
\begin{figure}{
   \def\svgwidth{4.8in}
\begingroup%
  \makeatletter%
  \providecommand\color[2][]{%
    \errmessage{(Inkscape) Color is used for the text in Inkscape, but the package 'color.sty' is not loaded}%
    \renewcommand\color[2][]{}%
  }%
  \providecommand\transparent[1]{%
    \errmessage{(Inkscape) Transparency is used (non-zero) for the text in Inkscape, but the package 'transparent.sty' is not loaded}%
    \renewcommand\transparent[1]{}%
  }%
  \providecommand\rotatebox[2]{#2}%
  \newcommand*\fsize{\dimexpr\f@size pt\relax}%
  \newcommand*\lineheight[1]{\fontsize{\fsize}{#1\fsize}\selectfont}%
  \ifx\svgwidth\undefined%
    \setlength{\unitlength}{373.27363046bp}%
    \ifx\svgscale\undefined%
      \relax%
    \else%
      \setlength{\unitlength}{\unitlength * \real{\svgscale}}%
    \fi%
  \else%
    \setlength{\unitlength}{\svgwidth}%
  \fi%
  \global\let\svgwidth\undefined%
  \global\let\svgscale\undefined%
  \makeatother%
  \begin{picture}(1,0.27342758)%
    \lineheight{1}%
    \setlength\tabcolsep{0pt}%
    \put(0,0){\includegraphics[width=\unitlength,page=1]{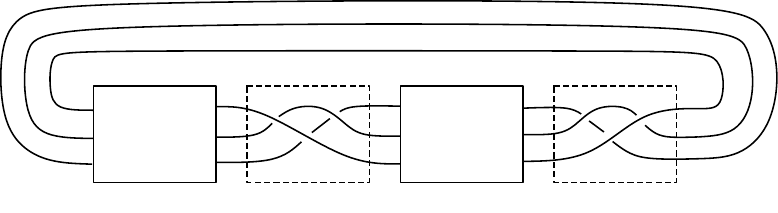}}%
    \put(0.18325918,0.09553507){\makebox(0,0)[lt]{\lineheight{1.25}\smash{\begin{tabular}[t]{l}$T$\end{tabular}}}}%
    \put(0.5815504,0.09553509){\makebox(0,0)[lt]{\lineheight{1.25}\smash{\begin{tabular}[t]{l}$T$\end{tabular}}}}%
    \put(0.37569202,0.00564061){\makebox(0,0)[lt]{\lineheight{1.25}\smash{\begin{tabular}[t]{l}$\Delta$\end{tabular}}}}%
    \put(0.7695081,0.00564056){\makebox(0,0)[lt]{\lineheight{1.25}\smash{\begin{tabular}[t]{l}$\Delta^{-1}$\end{tabular}}}}%
  \end{picture}%
\endgroup%

}
\caption{A planar diagram for a freely 2-periodic link $\tL$.}
\label{fig:tDprelim}
\end{figure}
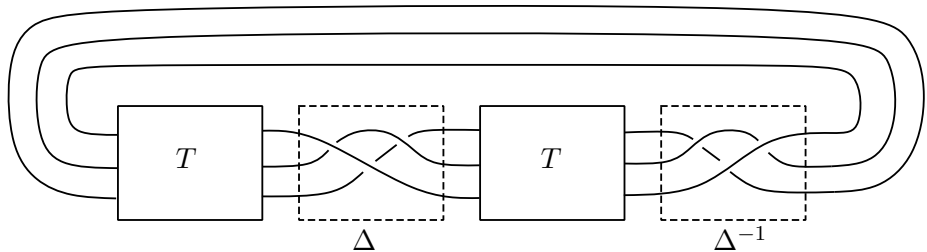

From here, we can get an even simpler planar diagram for $\tL$. Let $\flip(T)$ denote the result of rotating $T$ by $180^\circ$ about its middle horizontal axis; equivalently, we can reflect $T$ in its middle horizontal axis and then reverse all the crossings. (For example, when $T$ is a braid, $\flip(T)$ is obtained by applying the involution $\sigma_i \mapsto \sigma_{m-1-i}$ on the generators of the braid group $B_m$.) The simpler diagram for $\tL$ is the closure of the composition $T \circ \flip(T)$, as shown in Figure~\ref{fig:tD}. We denote this diagram by $\tD$.
\begin{figure}{
   \def\svgwidth{2.87in}
\begingroup%
  \makeatletter%
  \providecommand\color[2][]{%
    \errmessage{(Inkscape) Color is used for the text in Inkscape, but the package 'color.sty' is not loaded}%
    \renewcommand\color[2][]{}%
  }%
  \providecommand\transparent[1]{%
    \errmessage{(Inkscape) Transparency is used (non-zero) for the text in Inkscape, but the package 'transparent.sty' is not loaded}%
    \renewcommand\transparent[1]{}%
  }%
  \providecommand\rotatebox[2]{#2}%
  \newcommand*\fsize{\dimexpr\f@size pt\relax}%
  \newcommand*\lineheight[1]{\fontsize{\fsize}{#1\fsize}\selectfont}%
  \ifx\svgwidth\undefined%
    \setlength{\unitlength}{225.53810438bp}%
    \ifx\svgscale\undefined%
      \relax%
    \else%
      \setlength{\unitlength}{\unitlength * \real{\svgscale}}%
    \fi%
  \else%
    \setlength{\unitlength}{\svgwidth}%
  \fi%
  \global\let\svgwidth\undefined%
  \global\let\svgscale\undefined%
  \makeatother%
  \begin{picture}(1,0.3950674)%
    \lineheight{1}%
    \setlength\tabcolsep{0pt}%
    \put(0,0){\includegraphics[width=\unitlength,page=1]{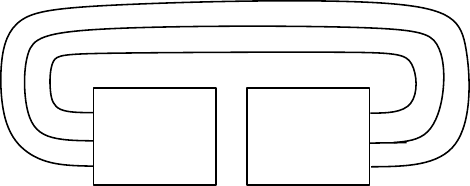}}%
    \put(0.3033005,0.09555992){\makebox(0,0)[lt]{\lineheight{1.25}\smash{\begin{tabular}[t]{l}$T$\end{tabular}}}}%
    \put(0.58228192,0.09062219){\makebox(0,0)[lt]{\lineheight{1.25}\smash{\begin{tabular}[t]{l}$\flip(T)$\end{tabular}}}}%
    \put(0,0){\includegraphics[width=\unitlength,page=2]{tD.pdf}}%
  \end{picture}%
\endgroup%

}
\caption{The simpler planar diagram $\tD$ for $\tL$.}
\label{fig:tD}
\end{figure}

\begin{lemma}
\label{lem:poslift}
If $L \subset \RP^3$ is a positive link, then its lift $\tL \subset S^3$ is also a positive link.
\end{lemma}

\begin{proof}
If $D$ is a positive diagram for $L$, then $\tD$ is a positive diagram for $\tL$.
\end{proof}

\begin{proposition}
\label{prop:abc}
Let $L \subset \RP^3$ be a link of class $\alpha \in \{0,1\}$. If $L$ satisfies any of the following three conditions:
\begin{enumerate}[(a)]
\item $L$ is local,
\item $L$ is of the form $U_1 \# L_\ell$ for some $L_\ell \subset S^3$, or
\item $L$ is positive,
\end{enumerate}
then we have
\begin{equation}
\label{eq:stl}
 s(\tL)= 2s(L) + \alpha - 1.
 \end{equation}
\end{proposition}

\begin{proof}
(a) When $L$ is local, its lift $\tL$ is the split disjoint union $L \sqcup L$ and we have $\alpha=0$. The relation $s(L \sqcup L)=2s(L)-1$ follows from Propositions~\ref{prop:disj unions and conn sums} (a) and \ref{prop:local links}.

(b) In this case $\alpha=1$ and $\tL = L_{\ell} \# L_{\ell}$. We get $s(L_{\ell} \# L_{\ell})=2s(U_1 \# L_{\ell})$ from Propositions~\ref{prop:disj unions and conn sums} (b), \ref{prop:local links}, and the fact that $s(U_1) = 0.$

(c) When $L$ is positive, so is $\tL$ by Lemma~\ref{lem:poslift}. Let $D$ be a positive diagram for $L$ with $n$ crossings, and with the oriented resolution $D_o$ having $k$ circles. Then $s(L) = n-r+1 = s(D_o) + n$ by Theorem~\ref{thm:positive} and Remark~\ref{rem:positive links}. For the positive diagram $\tD$ upstairs with oriented resolution $\tD_o$, we similarly have $s(\tL) = s(\tD_o) + 2n.$ Thus, it suffices to check that $s(\tD_o) = 2s(D_o) + \alpha - 1.$ Observe that $D_o$ is of the same class $\alpha$ as $L$, and it is either a local link (an unlink) or the connected sum of $U_1$ and an unlink. Using parts (a) and (b), we get the desired equality.
\end{proof}

\begin{remark}
The equality \eqref{eq:stl} fails for more general links $L \subset \RP^3$. For example, let $L$ be the projective closure of the negative full-twist on two strands:
\begin{center}
{
   \def\svgwidth{.6in}
\begingroup%
  \makeatletter%
  \providecommand\color[2][]{%
    \errmessage{(Inkscape) Color is used for the text in Inkscape, but the package 'color.sty' is not loaded}%
    \renewcommand\color[2][]{}%
  }%
  \providecommand\transparent[1]{%
    \errmessage{(Inkscape) Transparency is used (non-zero) for the text in Inkscape, but the package 'transparent.sty' is not loaded}%
    \renewcommand\transparent[1]{}%
  }%
  \providecommand\rotatebox[2]{#2}%
  \newcommand*\fsize{\dimexpr\f@size pt\relax}%
  \newcommand*\lineheight[1]{\fontsize{\fsize}{#1\fsize}\selectfont}%
  \ifx\svgwidth\undefined%
    \setlength{\unitlength}{169.07241387bp}%
    \ifx\svgscale\undefined%
      \relax%
    \else%
      \setlength{\unitlength}{\unitlength * \real{\svgscale}}%
    \fi%
  \else%
    \setlength{\unitlength}{\svgwidth}%
  \fi%
  \global\let\svgwidth\undefined%
  \global\let\svgscale\undefined%
  \makeatother%
  \begin{picture}(1,0.72654819)%
    \lineheight{1}%
    \setlength\tabcolsep{0pt}%
    \put(0,0){\includegraphics[width=\unitlength,page=1]{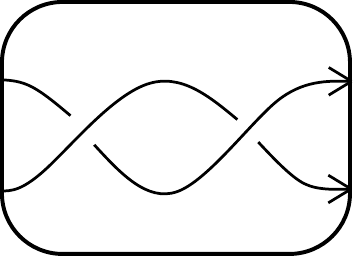}}%
  \end{picture}%
\endgroup%

}
\end{center}
Then, one can check that $s(L)=s(\tL)=-1$, so $s(\tL) \neq 2s(L) -1.$
\end{remark}

It is harder to find examples of class-1 knots $K \subset \RP^3$ for which $s(\tK) \neq 2s(K)$. In the introduction we gave two such examples, $K_1$ and $K_2$, for which we have
$$s(K_1)=0, \ \ s(\tK_1)=2; \ \ \ \ s(K_2)=2, \ \ s(\tK_2)=2.$$
 To find the knots $K_1$ and $K_2$, we searched through the projective closures of braids on 5 strands and at most 8 crossings. (A similar search through braids on 3 strands did not give any interesting examples.) To compute $s(\tK)$ we used the Mathematica program {\em KnotTheory'}  \cite{Mathematica, KnotTheory}. To compute $s(K)$ we used the older {\tt Categorification.m} program from \cite{Categorification}. This has the advantage that it does not essentially use the fact that it is meant for knots in $S^3$; one can apply it just as well to knots in $\RP^3$, by plugging in the PD code from a diagram in $\RP^2$ (numbering its edges and listing its crossings as if it were a planar diagram). While {\tt Categorification.m} is much slower than the newer {\tt UniversalKh} package from {\em KnotTheory'}, it can still easily compute the $s$-invariants of our projective diagrams with up to $8$ crossings. 
 
We identified the lifts $\tK_1$ and $\tK_2$ to be the knots $12n403$ and $14n14256$ using {\em SnapPy} \cite{SnapPy}. Furthermore, to compute their slice genera and standardly equivariant slice genera, we proceeded as follows. 

The signature of $\tK_1$ can be computed to be $4$, and therefore $2=\sigma(\tK_1)/2 \leq g_4(\tK_1)$. On the other hand, it is not hard to see that changing the circled crossing in the projective diagram for $K_1$ from Figure~\ref{fig:crossings} (a) results in the unknot $U_1$. This produces a genus $1$ cobordism in $\RP^3$ between $K_1$ and $U_1$, implying that $g_4(K_1) \leq 1.$ Hence
$$ 2 \leq g_4(\tK_1) \leq \gequiv(\tK_1) = 2 g_4(K_1) \leq 2,$$
so all these inequalities are equalities.

\begin{figure}
{
(a) \ \  \def\svgwidth{1.5in}
   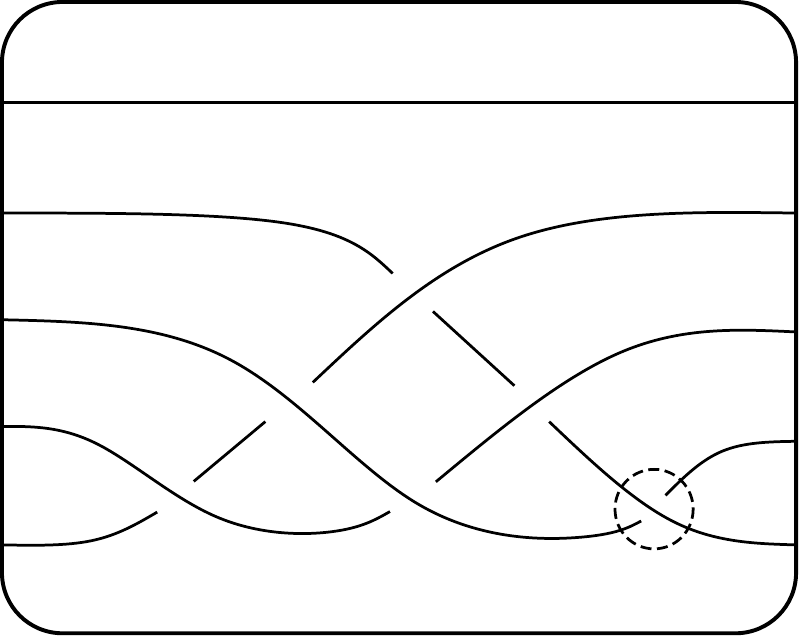
  \ \ \ \ \ \ \  \ \ (b) \ \  \def\svgwidth{1.5in}
  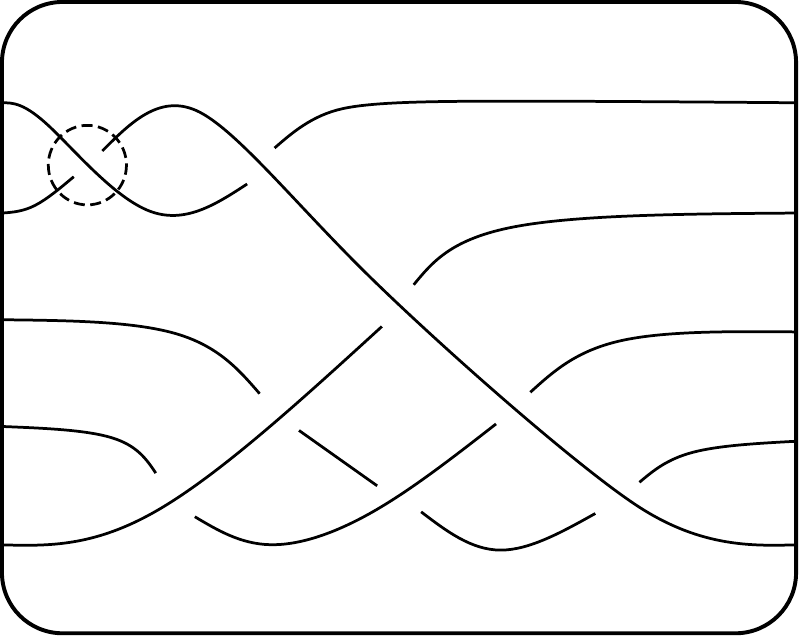
}
\caption{Crossing changes that take $K_1$ and $K_2$ into the class-1 unknot $\tU_1$.}
\label{fig:crossings}
\end{figure}

For $K_2$, changing the circled crossing in Figure~\ref{fig:crossings} (b) yields the unknot $U_1$, so $g_4(K_2) \leq 1$. Interestingly, if we change the same crossing in the diagram $\tD$ for $\tK_2$ shown in Figure~\ref{fig:tK2crossing}, we get the unknot in $S^3$, which implies that $g_4(\tK_2) \leq 1.$ The inequalities $1=s(K_2)/2 \leq g_4(K_2)$ and $1=s(\tK_2)/2 \leq g_4(\tK_2)$ imply that $g_4(K_2) = \gequiv(\tK_2)/2 = 1$ and $g_4(\tK_2) = 1.$
\begin{figure}
{
   \def\svgwidth{5in}
   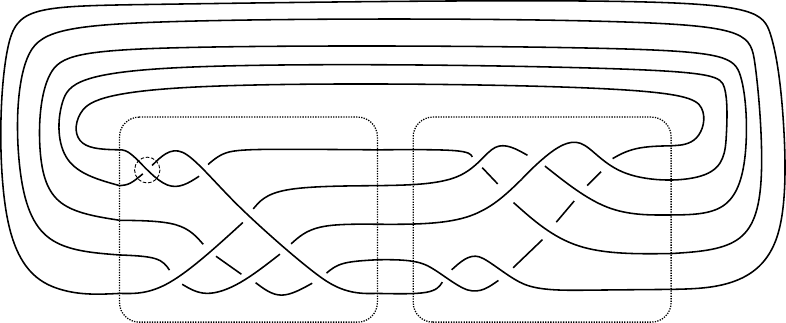
}
\caption{A crossing change on $\tK_2$ producing the unknot.}
\label{fig:tK2crossing}
\end{figure}

In the above arguments we used the following result, which was alluded to in the introduction. For completeness, we present a proof.

\begin{lemma}
Let $K \subset \RP^3$ be a class-1 knot, and $\tK$ its lift to $S^3$. Then, the standardly equivariant genus of $\tK$ is equal to twice the slice genus of $K$: $\gequiv(\tK) = 2g_4(K)$.  
\end{lemma}

\begin{proof}
Given an oriented cobordism $W \subset I \times \RP^3$ of genus $g_4(K)$ from $K$ to $U_1$, its lift to the double cover gives a standardly equivariant cobordism $\tW \subset I \times S^3$ from $\tK$ to the unknot $U=S^1 \subset S^3$; we can fill this with the standard equivariant disk $B^2 \subset B^4$ to get an equivariant surface $\tSigma \subset B^4$ with boundary $\tK$. An Euler characteristic computation shows that $\tSigma$ has genus $2g_4(K)$. Thus, $\gequiv(\tK) \leq 2g_4(K)$.

Conversely, consider a (standardly) equivariant surface $\tSigma \subset B^4$ with boundary $\tK$, of genus $\gequiv(\tK)$. First, note that $\tSigma$ must contain the origin $0$, because otherwise it would double cover a surface $\Sigma \subset I\times \RP^3$ with boundary $K$, which would contradict the fact that $[K] \neq 0 \in H_1(\RP^3; \Z)$. Near the origin $0$, the surface $\Sigma$ must look locally like its tangent space $T\Sigma$ with a nontrivial $\Z/2$-action, which means that we can excise a ball around $0 \in B^4$ to get a standardly equivariant cobordism $\tW\subset I \times S^3$ from $\tK$ to the equivariant unknot $\tU_1$. This is the double cover of some cobordism $W$ from $K$ to $U_1$, whose genus must be $\gequiv(\tK)/2$. This implies that $g_4(K) \leq \gequiv(\tK)/2$.
\end{proof}

Finally, we prove the application about standardly equivariant concordance.

\begin{proof}[Proof of Theorem~\ref{thm:eqconc}]
The knots $\tK$ and $\tK' = (-\tK) \# \tK \# \tK$ are clearly concordant, because $(-\tK) \# \tK$ is slice. On the other hand, the existence of a standardly equivariant concordance between $\tK$ and $\tK'$ is equivalent to that of a concordance (an annular cobordism)  in $I \times \RP^3$ from $K$ to $K'=(-K) \# \tK$. The latter would imply that  $s(K)=s(K')$, using Theorem~\ref{thm:BW}. However, we chose our knot $K$ so that $s(\tK) \neq 2s(K)$ and hence $$s(K') = s(-K) + s(\tK) = -s(K) + s(\tK) \neq s(K).$$
Here, we made use of Theorem~\ref{thm:local} and Proposition~\ref{prop:mirrors}.
\end{proof}

\subsection{Open problems}
In \cite[Question 1]{BM}, Boyle and Musyt ask if there are freely periodic slice knots that are not equivariantly slice. A weaker version of their question is:
\begin{question}
\label{q:slice}
Does there exist a freely 2-periodic slice knot  $\tK \subset S^3$ that is not standardly equivariantly slice, i.e., such that the corresponding class-1 knot $K \subset \RP^3$ is not concordant to $U_1$?
\end{question}

Theorem~\ref{thm:eqconc} provides examples of a similar flavor (with concordance instead of sliceness), but does not quite answer Question~\ref{q:slice}. One may be tempted to consider the direct sum $(-\tK) \# \tK'$ (with $\tK$ and $\tK'$ as in Theorem~\ref{thm:eqconc}), but in general the connected sum of two freely 2-periodic knots is not freely 2-periodic.  

Nevertheless, in principle the $s$-invariant constructed in this paper could be useful in answering Question~\ref{q:slice}: one needs to find a class-1 knot $K \subset \RP^3$ such that $s(K) \neq 0$ but $\tK$ is slice.

Question~\ref{q:slice}, along with the results in this paper, provides an impetus for a further study of concordance of knots in $\RP^3$. Apart from $s$, there are for example classical concordance invariants such as the $d$-signatures for $d$ odd \cite{Gilmer}, as well as invariants from knot Floer homology, such as Raoux's $\tau_{\mathfrak{s}}$ invariants \cite{Raoux, HR}. It would be interesting to see if one could recover results such as Theorem~\ref{thm:eqconc} using $\tau_{\mathfrak{s}}$.

In a different direction, while Theorem~\ref{thm:BW} gives genus bounds for surfaces in $I \times \RP^3$, one may wonder about the genus of surfaces in 4-manifolds with a single boundary $\RP^3$. A natural candidate is $DTS^2$, the disk bundle associated to the tangent bundle to $S^2$. We conjecture the following:
\begin{conjecture}
\label{conj:DT}
Let $\Sigma \subset DTS^2$ be an oriented, smoothly and properly embedded surface of genus $g$ with boundary $K \subset \del(DTS^2) = \RP^3$. Suppose that $\Sigma$ is null-homologous, i.e., its relative homology class in $H_2(DTS^2, \RP^3; \Z) \cong \Z$ is zero. Then:
$$ -s(K) \leq 2g.$$
\end{conjecture}

This is similar in spirit to Theorem 1.15 in \cite{MMSW}, which gives genus bounds for null-homologous  surfaces in $S^1 \times B^3$ and $B^2 \times S^2$ with boundary a knot $K \subset S^1 \times S^2$; and also to Corollary 1.9 in \cite{MMSW}, which gives  genus bounds for null-homologous surfaces in $(\#^t \overline{\mathbb{CP}^2}) \setminus B^4$  with boundary a knot $K \subset S^3.$ Those results are proved by considering the intersection of the surface with the ``core'' of the 4-manifold, and obtaining a cobordism between $K$ and a link from a fixed infinite family. The same idea can be used to approach Conjecture~\ref{conj:DT}. In this case the surface in $DTS^2$ gives a cobordism in $I \times \RP^3$  relating $-K$ to a link $H_p \subset \RP^3$ obtained as the projective closure of a positive half-twist on $p$ positively oriented and $p$ negatively oriented strands. (A picture of $H_2$ is given in Figure~\ref{fig:H2}.) Conjecture~\ref{conj:DT} would follow from Theorem~\ref{thm:BW} if one could prove that $s(H_p) = 1-2p$. 
\begin{figure}
{
   \def\svgwidth{1.5in}
   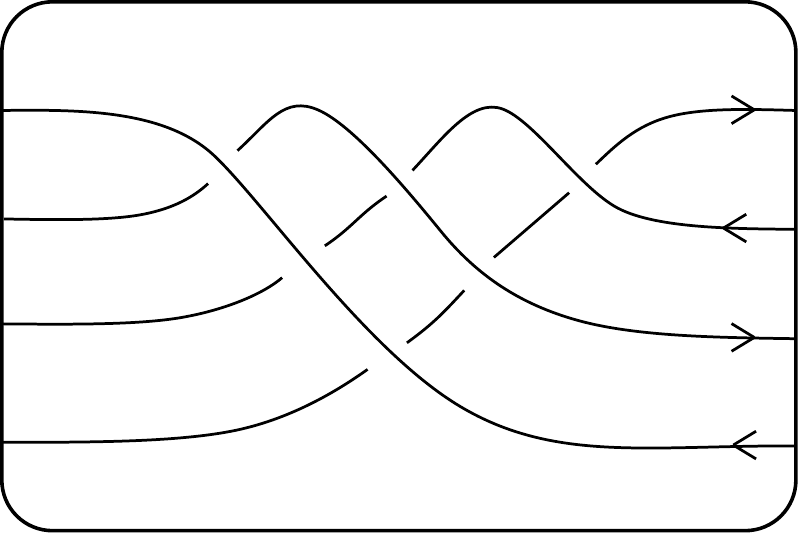
}
\caption{The projective closure $H_2$ of the half-twist on $4$ balanced strands.}
\label{fig:H2}
\end{figure}

\bibliographystyle{custom}
\bibliography{biblio}

\end{document}